%% file: eqmotHA.tex
\documentclass[a4paper,reqno]{amsart}

\input{header.tex}

\newcommand{\trl}[1]{#1^{\triangleleft}}
\newcommand{\trr}[1]{#1^{\hspace{.5pt}\triangleright}}
\newcommand{\CHM}[1]{\operatorname{Chow}(#1)}
\newcommand{\CHMeff}[1]{\operatorname{Chow}^{\operatorname{eff}}(#1)}
\newcommand{\Hall}{\operatorname{H}}
\newcommand{\modfp}{\operatorname{mod}}
\renewcommand{\s}[1]{#1^{\operatorname{s}}}
\renewcommand{\ss}[1]{#1^{\operatorname{ss}}}

\begin{document}

\title{Equivariant motivic Hall algebras}
\author{Thomas Poguntke}
\address{Hausdorff Center for Mathematics \\ Endenicher Allee 62, 53115 Bonn, Germany}
\email{thomas.poguntke@hcm.uni-bonn.de}


\begin{abstract}
{We introduce a generalization of Joyce's motivic Hall algebra by combining it with Green's parabolic induction product, as well as a non-archimedean variant of it. In the construction, we follow Dyckerhoff-Kapranov's formalism of $2$-Segal objects and their transferred algebra structures. Our main result is the existence of an integration map under any suitable transfer theory, of course including the (analytic) equivariant motivic one. This allows us to study Harder-Narasimhan recursion formulas in new cases.}
\end{abstract}

\date{\today}
\maketitle

{\hypersetup{linkbordercolor=white} \tableofcontents}

\section{Introduction}\label{sec1}

The theory of $2$-Segal objects has led to a systematic framework to study multi-valued and convolution-type algebraic structures \cite{HSS1}, including Hall and Hecke algebras. In this work, we construct a new type of Hall algebra based on this formalism, as a combination of those introduced in \cite{Green1} and \cite{Joy2}. We are motivated by the following algebraic and arithmetic geometric applications.

Let $\E$ be a quasi-abelian category linear over a field $K$, with $\sum_{i \geq 0} \dim_K \Ext^i_{\E}(-,-) < \infty$. Reineke \cite{Rei} introduced the following methodology of counting $\FF_q$-points of quiver moduli spaces in order to infer their Betti numbers over $\mathbb C$. If $K=\FF_q$ and $\E$ is hereditary, there is a morphism on the (completed) Ringel-Hall algebra with values in the (completed) group ring of the Grothendieck group twisted by the Euler form,
\begin{equation}\label{intmap}
\int_{\E}\!: \QQ[[\pi_0(\E^{\simeq})]] \longto \QQ^{\langle\chi^{\op}\rangle}[[K_0(\E)]].
\end{equation}
This allows us to ``integrate" equations in the Hall algebra, arising directly from categorical structures on $\E$ (or geometric structures on its moduli stack of objects), and translate them into interesting identities in more managable rings (like $\QQ$).

In particular, slope filtrations provide a rich such source (cf. \cite{An}). They have their origin in the construction of moduli spaces of vector bundles on curves by means of geometric invariant theory \cite{Mum} (also cf. \cite{King} for the case of moduli spaces of quiver representations). Exploiting these categorical structures to recursively compute the cohomology of the corresponding moduli spaces goes back to Harder-Narasimhan \cite{HN}; for a modern work in the subject, see \cite{MoRe2} for example.

The existence and uniqueness of Harder-Narasimhan filtrations in $\E$ is expressed in the Hall algebra as the single equation
\begin{equation}\label{HNequ}
\mathbbm 1_{\pi_0(\E^{\simeq})} = \sum_{\lambda_1 > \cdots > \lambda_n} \mathbbm 1_{\pi_0(\E_{\lambda_1}^{\operatorname{ss} \simeq})} * \cdots * \mathbbm 1_{\pi_0(\E_{\lambda_n}^{\operatorname{ss} \simeq})} \in \QQ[[\pi_0(\E^{\simeq})]].
\end{equation}
Joyce \cite{Joy2} refines Reineke's point count to arbitrary motivic measures of moduli stacks (of quiver representations or vector bundles) over an arbitrary field $K$. The idea is to replace the number of points $q = \#\mathbb A^1(\FF_q)$ by the affine line itself, as an element $\LL = [\mathbb A^1_K] \in K_0(\Var K)$ of the Grothendieck ring of varieties; in fact, since the Hall algebra crucially relies on keeping track of automorphisms, the Grothendieck ring of algebraic stacks (cf. \cite{Br} and \cite{Toen})
\[K_0(\Sta K) \isom K_0(\Var K)[\LL^{-1},(\LL^n-1)^{-1} \mid n\in\NN].\]
Just as the Ringel-Hall algebra arises as a transfer of the Waldhausen construction \cite{HSS1}, the motivic Hall algebra $\H(\E) = K_0(\Sta \M_{\E})$ is the convolution algebra attached to its stacky version $\S(\E)$. In our non-derived context, there is no suitable general notion of this in the literature as of yet, and we propose the following simple and appropriately functorial construction. The moduli stack of objects of $\E$ is defined as
\[\M_{\E}\!: \Aff_K^{\et} \longto \Grpd,\ \Spec(R) \longmapsto (\E \hatotimes_K R)^{\simeq},\]
where $- \otimes_K -$ means the $\Hom$-wise tensor product and $\hat{\phantom{a}}$ the $K$-linear Cauchy completion. It might be possible to adapt the approach in \cite{CG} to our (non-cocomplete) setting, but this would complicate the following crucial finiteness result. The upcoming work on stack-theoretic GIT of Halpern-Leistner and Heinloth will feature a construction somewhat similar to ours.

\begin{thm}\label{mainthm6}
The functor $\M_{\E}$ defines an algebraic stack, locally of finite type over $K$.
\end{thm}

The proof relies heavily on the analogous result in \cite{TV}, by locally comparing $\M_{\E}$ to their moduli stack of modules (which provides a suitable moduli stack of objects only in the derived setting). Joyce's analogue of the integration map \eqref{intmap} from \cite{Joy2} and our generalization of it (see below) again provides a Harder-Narasimhan recursion,
\[[\M_{\alpha}] = \sum_{\substack{\alpha_1 > \cdots > \alpha_n \\ \alpha_1 + \cdots + \alpha_n = \alpha}} \LL^{-\sum_{i<j}\chi^{\op}(\alpha_j,\alpha_i)} [\ss\M_{\alpha_1}]\cdots[\ss\M_{\alpha_n}] \in K_0(\Sta K),\]
in degree $\alpha\in K_0(\E)$. Reineke and Joyce provide naive, but combinatorially involved inversion formulas. Even more sophisticated approaches are due to Zagier \cite{Zag} and Laumon-Rapoport \cite{LR}, following the suggestion of Kottwitz to apply the Langlands lemma on Eisenstein series. This is further pursued in \cite{DOR} to compute the cohomology of $p$-adic period domains, which do not fit into Joyce's formalism; our motivation is to accommodate for that.

The cohomology of a smooth projective variety with good reduction over a finite totally ramified extension of complete discretely valued fields of mixed characteristic with perfect residue field carries the structure of a semistable filtered $F$-isocrystal (equivalently, a crystalline Galois representation). This is the $p$-adic analogue of a Hodge structure, so their moduli spaces are the analogues of period domains (cf. \cite{RZ}, and \cite{Har} for the equal characteristic case).

We generalize Joyce's approach in order to include this setting, where we parametrize structures on a fixed $K$-rational object, and compute invariants equivariant with respect to its automorphism group. For example, the Euler characteristic of a flag variety is not just a number (or even a virtual Galois representation), but a virtual representation of $\GL_n(K)$ (possibly additionally with the commuting action of the absolute Galois group).

More precisely, we start with a pair of exact isofibrations $\D \xto{\nu} \B \xfrom{\omega} \E$, and the moduli stack $\ubar{\D} \times_{\M_{\B}} \M_{\E}$ classifies structured objects (in $\E$), where the underlying object (in $\B$) has a $K$-rational structure (in $\D$; which may also be $\FF_1$- instead of $K$-linear, as for the "period domains over $\FF_1$" in \cite{DOR}). Beyond $p$-adic period domains and their variants, many further examples arise from fibre functors of (quasi-)Tannakian categories over $K$, as well as forgetful functors on categories of filtered objects in any $\B$ (e.g. cf. \cite{FilQ}).

Now assume that $\D$ is split (proto-)exact, as we want to think of it rather as a collection of structure groups. The $\D$-equivariant variant of $K_0(\Sta -)$ defines the equivariant motivic Hall algebra
\[\H^{\D}(\E) := K_0^{\D}(\Sta \M_{\E}) \isom \bigoplus_{D\in\pi_0(\D^{\simeq})} K_0^{\Aut(D)}(\Sta \M_{\E}),\]
with parabolic induction product between summands. This is different from the equivariant motivic Hall algebra in \cite{emHA}, which is defined for a good action of an abelian variety $A$ on a smooth projective CY$3$-fold $X$ over $K=\CC$, inducing an action on $\M_{\E} = \Coh_X$, essentially by replacing $\St K$ with $\St BA$ in the construction of the usual motivic Hall algebra. Here is our generalization of Joyce's integration (the case $\D=0$).

\begin{thm}\label{mainthm7}
If $\E$ is hereditary, there is an algebra morphism
\[\int_\E^{\D}\!: \hat\H^{\D}(\E) \longto K_0^{\D}(\Sta K)^{\pair{\chi^{\op}}}[[K_0(\E)]],\]
induced by the natural map of simplicial stacks $\S(\E) \to \ubar{K_0(S(\E))}$.
\end{thm}

In fact, the result holds for any suitable transfer theory, in particular the non-archimedean analytic version $K_0^{\D}(\AnSta \M_{\E}^{\an})$ of the equivariant motivic Hall algebra, based on the theory of analytic stacks from \cite{Ul}. This is to circumvent the non-existence of the Harder-Narasimhan recursion for $[\ubar{\D} \times_{\M_{\B}} \M_{\E}] \in \hat\H^{\D}(\E)$ when $K$ is infinite; namely, finiteness of the $\Hom$ spaces of $\D$ is replaced by compactness of the analytified moduli space attached to $\D$.

If $\E$ is a (full abelian subcategory of a) CY$3$-category over $\CC$, there is still a conjectural integration map \cite{KS}, dependent on glueing together the twist from motivic vanishing cycles of functions whose critical loci locally describe the moduli stack.

When $\E = \Coh(X)$, another approach is Bridgeland's Poisson algebra structure on the submodule generated by varieties \cite{Br}; this also works in the setting of \cite{emHA} to study reduced motivic Donaldson-Thomas invariants. It is a very interesting question whether the same is true for our construction, defining equivariant motivic DT-invariants.

Finally, there is a large body of work on the cohomological Hall algebra approach to motivic Donaldson-Thomas theory, which originates in \cite{CoHA}, and has been developed as well as surveyed for instance in \cite{Mei} and \cite{DMDT}, as well as references therein.

Let us briefly outline the structure of the paper. In the first section, we introduce a naive, but general and functorial construction of the moduli stack of objects in a suitable linear category. This will be used throughout for all the geometric examples of Hall algebras we consider. In fact, $\mathsection$\ref{sec7} explains how to extract genuine algebras from unital $2$-Segal objects, in particular the moduli stack version of Waldhausen's $S$-construction.

In $\mathsection$\ref{sec8}, we apply this theory to introduce the main objects of interest, the equivariant motivic Hall algebra and its non-archimedean analytic variant, while the next section contains our main result on the existence of an integration map in a rather general context, including the constructions of the previous section.

As an application, in $\mathsection$\ref{sec10}, we recall the formalism of slope filtrations and produce new cases of the well-studied Harder-Narasimhan recursion and its various inversions.

Finally, $\mathsection$\ref{sec11} outlines the analogue of our theory for Hall modules, including a proof sketch for the existence of integration maps.

\addtocontents{toc}{\skipTOCentry}
\section*{Acknowledgements}
I would like to thank T. Dyckerhoff for suggesting this topic, his constant interest and invaluable contributions. I would also like to thank M. Young for interesting discussions about Hall modules, W. Stern for helpful suggestions, as well as M. Penney and T. Walde for fruitful discussions. Finally, I would like to thank M. Kapranov and S. Mozgovoy for interesting discussions during a visit to Dublin.

\section{Moduli stacks of objects in locally proper categories}\label{sec6}

In this section, we prepare for many of the forthcoming examples of algebraic incarnations of higher Segal objects by proposing a simple definition of the moduli stack of objects in a locally proper Cauchy complete linear category, and by studying its basic properties. While naive, the construction is suitable for our purposes in view of its functoriality.

Let $K$ be a commutative ring (we save some notation here since in examples, $K$ will generally be a field). By a stack over $K$, we will always mean a stack in groupoids on the big \'etale site $\Aff_K = \Aff_K^{\et}$ of affine schemes over $K$. Their category is denoted by $\St K$. Let $\E$ be a $K$-linear category.

\begin{defn}
For a $K$-linear category $\D$, the tensor product $\D \otimes_K \E$ over $K$ has the same objects as $\D \times \E$, but with morphisms defined by the tensor product,
\[\Hom_{\D \otimes_K \E}((A,M),(B,N)) = \Hom_{\D}(M,N) \otimes_K \Hom_{\E}(A,B),\]
and composition given by the tensor product of the composition in $\D$ and $\E$.
\end{defn}

\begin{ex}
Let $\E$ be the category of abelian varieties over a field. Then
\[\E_{\QQ} = \E \otimes_{\ZZ} \QQ\]
is the isogeny category of $\E$. Similarly for $\D \otimes_{\O_F} F$, where $\O_F \subseteq F$ is the ring of integers in a global function field and $\D$ is the category of Drinfeld modules over $\O_F$.
\end{ex}

We denote by $\Grpd$ the category of groupoids.

\begin{defn}\label{maindefn}
The moduli stack of objects in $\E$ is the stack over $K$ defined by
\[\M_{\E}\!: \Aff_K \longto \Grpd,\ S = \Spec(R) \longmapsto (\E \hatotimes_K R)^{\simeq},\]
where on the right-hand side, $R$ is the category with one object with endomorphism ring $R$, and $- \hatotimes_K -$ denotes the $K$-linear Cauchy completion of the tensor product.
\end{defn}

\begin{rem}\label{cauchyadj}
If $\Spec(R') \onto \Spec(R)$ is an \'etale cover in $\Aff_K$, then the equivalence
\begin{equation}\label{naivedesc}
  \begin{tikzcd}[column sep=1.5em]
    \E_R := \E \otimes_K R \ar[r, "{\sim}"] & \lim\big(\E_{R'} \ar[r, shift left=0.8] \ar[r, shift right=0.8] & \E_{R' \otimes_R R'} \ar[r] \ar[r, shift left=1.6] \ar[r, shift right=1.6] & \E_{R' \otimes_R R' \otimes_R R'}\big)
  \end{tikzcd}
\end{equation}
follows from faithfully flat descent. On the other hand, by its universal property, the Cauchy completion is a left adjoint. But it also commutes with finite $2$-limits (cf. \cite{BLS}, $\mathsection$1.2), and therefore, \eqref{naivedesc} implies that $\M_{\E}$ is a stack. Note that $(-)^{\simeq}$ is right adjoint to $\Grpd \into \Cat$, hence commutes with limits.

Note that $\M_{\E}$ is functorial in $\E$ with respect to arbitrary functors. In particular, applied to the unit $\E \to \hat\E$, this induces an equivalence $\M_{\E} \isoto \M_{\hat\E}$. Thus, while Definition \ref{maindefn} makes sense in general, we may assume $\E$ to be Cauchy complete henceforth.
\end{rem}

\begin{defn}
Let $S$ be a qcqs scheme, and $\qcoh_S$ the category of finitely presented quasi-coherent $\O_S$-modules. We denote by $\vect_S$ its full subcategory of locally free sheaves of finite rank. The dual of an object $F\in\qcoh_S$ is defined to be the quasi-coherent $\O_S$-module
\[F^{\vee} = \HHom_{\O_S}(F,\O_S).\]
The sheaf $F$ is called reflexive, if the canonical map to its bidual is an isomorphism,
\[F \isoto F^{\vee\vee}.\]
The category $\refl_S$ is the full subcategory of reflexive $\O_S$-modules
\[\vect_S \subseteq \refl_S \subseteq \qcoh_S.\]
Similarly, if $E$ is a $K$-algebra, we write $\Mod_E$ for the category of (right) $E$-modules, and
\[\proj_E \subseteq \modfp_E\]
for its full subcategories of finitely presented flat, resp. finitely presented, modules.
\end{defn}

\begin{defn}
We say that $\E$ is locally flat if for all $A,B\in\E$, the $K$-module $\Hom(A,B)$ is flat. The category $\E$ is called locally proper if $\Hom(A,B) \in \proj_K$ for all $A,B\in\E$.
\end{defn}

\begin{rem}
The category $\proj_R$ is the Cauchy completion of $R$ itself. This can be thought of as a two-step process, where completion under (finite) direct sums yields the category of finite free $R$-modules, the idempotent completion of which is $\proj_R$. All of this occurs inside the category $\Fun_K(R^{\op},\Mod_K) \isom \Mod_R$ of arbitrary $R$-modules.

Write $S = \Spec(R)$. Because of the above, it seems reasonable to think of $\E \hatotimes_K R$ as parametrizing ``flat $S$-families of objects in $\E$'', at least if $\E$ is locally proper. For $s \in S$, there is an induced ``fibre functor''
\[\E \hatotimes_K R \longto \E \hatotimes_K \kappa(s).\]
\end{rem}

\begin{defn}
The moduli stack of modules of a $K$-linear category $\E$ is defined by
\[\U_{\E}\!: \Aff_K \longto \Grpd,\ S \longmapsto \Fun_K(\E^{\op}, \vect_S)^{\simeq}\]
as a functor of points.
\end{defn}

\begin{rem}
The stack $\U_{\E}$ was introduced by To\"en-Vaqui\'e \cite{TV}, who show that the analogous definition for smooth proper dg-categories actually provides a reasonable (higher) moduli stack of objects, unlike in our linear setting -- the notation is chosen to suggest $\U_{\E}$ be ``underived''.

Let $S = \Spec(R) \in \Aff_K$. Similarly to Remark \ref{cauchyadj}, the assignment $\U_{\E}$ is a stack over $K$, by faithfully flat descent for $\vect_S$, now using the fact that $\Fun_K(\E^{\op}, -)$ is left exact.

If $\E$ is locally proper, then the Yoneda embedding induces a fully faithful functor
\[\E \otimes_K R \longinto \Fun_R(\E_R^{\op}, \proj_R) \isom \Fun_K(\E^{\op}, \proj_R).\]
This defines a representable morphism of stacks over $K$,
\begin{equation}\label{NMU}
\eta\!: \M_{\E} \longto \U_{\E}
\end{equation}
since $\proj_R$ is Cauchy complete, and hence so is $\Fun_K(\E^{\op}, \proj_R)$.
\end{rem}

We consider algebraic stacks in the sense of \cite{Stacks}, Definition \href{http://stacks.math.columbia.edu/tag/026O}{026O}, with the \'etale in place of the fppf-topology; these definitions are essentially equivalent by \textit{op.cit.}, Lemma \href{http://stacks.math.columbia.edu/tag/04X2}{04X2}.

\begin{ex}\label{Bunex}
Let $X$ be a proper scheme over $K$. Then the moduli stack of vector bundles
\[\Bun_X = \coprod_{d\in\NN} \Bun_{X,\GL_d}\]
is an algebraic stack, locally of finite presentation over $K$. Indeed, let $\Coh_X$ be the stack assigning to $S \in \Aff_K$ the groupoid of $S$-flat finitely presented quasi-coherent sheaves on $X \times_K S$. Then by \cite{Stacks}, Theorem \href{http://stacks.math.columbia.edu/tag/09DS}{09DS}, $\Coh_X$ has the desired properties (also cf. \cite{LMB}, Th\'eor\`eme 4.6.2.1; \cite{Bun}, $\mathsection$4.4). But $\Bun_X \subseteq \Coh_X$ is an open substack.

Now denote by $\E = \vect_X$. Then the morphism\eqref{NMU} factors as
\begin{equation}\label{Bunfact}
  \begin{tikzcd}[column sep=tiny]
    & \M_{\E} \ar[rr, "{\eta}"] \ar[dr, "{\pr_X^*}"'] & & \U_{\E} & \ar[l, phantom, "{\ni}"] \Hom\!\mathrlap{{}_{\O_{X \times S}}(\pr_X^*(-),E)} & \ \ \ \ \ \ \ \ \ \ \ \ \ \ \ \ \ \ \ \\
    & & \Bun_X \ar[ur, "{\tilde{\eta}}"'] & \ar[l, phantom, "{\ni}"] E, \ar[ur, mapsto] & \mathrlap{\mbox{on } S\mbox{-points.}} & \ \ \ \ \ \ \ \ \ \ \ \ \ \ \ \ \ \ \ 
  \end{tikzcd}
\end{equation}
It is clear that neither $\eta$ nor $\tilde{\eta}$ is an equivalence. However, we infer that the map
\[\pr_X^*\!: \M_{\E} \isoto \Bun_X\]
is a natural equivalence, as $\tilde{\eta}$ maps into $\M_{\E}$.

If $X$ is flat over $K$, we can consider the full subcategory $\qcoh_X^{\flat} \subseteq \qcoh_X$ of $K$-flat sheaves to get a monomorphism $\M_{\qcoh_X^{\flat}} \into \Coh_X$. At least if $K$ is a field, this is an equivalence
\[\M_{\qcoh_X} \isoto \Coh_X.\]
\end{ex}

\begin{ex}
Suppose that $K$ is a field, and let $\E = \rep_K(Q)$ be the category of finite dimensional representations over $K$ of a finite (connected) quiver $Q$. For each $\ubar d \in \NN^{Q_0}$, set
\[\VV_{Q,\ubar d} = \Spec(\Sym(V_{Q,\ubar d}^{\vee})),\]
where
\[V_{Q,\ubar d} = \bigoplus_{\alpha\in Q_1} \Hom_K(K^{d_{s(\alpha)}}, K^{d_{t(\alpha)}}).\]
Each $\VV_{Q,\ubar d}$ carries an action of the corresponding general linear group by change of basis,
\[g.(\phi_{\alpha})_{\alpha\in Q_1} = (g_{t(\alpha)}\phi_{\alpha}g_{s(\alpha)}^{-1})_{\alpha\in Q_1} \mbox{ for } g \in \GL_{\ubar d}(R),\]
where $\GL_{\ubar d} = \prod_{i\in Q_0} \GL_{d_i}$ and $\phi_{\alpha}\in \Hom_K(K^{d_{s(\alpha)}}, K^{d_{t(\alpha)}}) \otimes_K R$. Then the moduli stack
\[\M'_{\E} := \coprod_{\ubar d \in \NN^{Q_0}} [\VV_{Q,\ubar d}/\GL_{\ubar d}]\]
parametrizes representations of $Q$ in $\vect_S$ (of constant rank on $S$), for $S = \Spec(R) \in \Aff_K$. Thus, the obvious map yields an equivalence of algebraic stacks locally of finite type
\[\M_{\E} \isoto \M'_{\E}.\]
On the other hand, the morphism $\eta'\!: \M'_{\E} \to \U_{\E}$ induced by \eqref{NMU} is only an equivalence if $Q$ is discrete. Otherwise, we may assume $Q = A_2$. Let $F\!: \E^{\op} \to \vect_K$ be a ($K$-linear) functor, and let $E = (K \into K^2)$ denote the regular representation. If $P\in\E$ is a projective $E$-module, then $F(P) \isom \Hom_{\E}(P,F(E))$, since the restricted Yoneda embedding is an equivalence, as
\[\Fun(\proj_E^{\op},\vect_K) \isoto \Fun(E^{\op},\vect_K),\ f \mapsto f(E).\]
Now, if F is representable, it is left exact. By the standard projective resolution, it follows that it is indeed represented by $F(E)$. But consider the functor $F := \Hom_{\E}(H(-),E)$, where
\[H(V \xto{\phi} W) = (0 \to \ker(\phi)).\]
Then $F(E) = 0$ (indeed, $F|_{\proj_E} = 0$), but $F \neq 0$ otherwise, hence it is not representable.

Instead, if $\hat Q$ is the category of finitely presented projective objects in the category of all representations of $Q$ over $K$, then $\U_{\hat Q} \isom \M_{\E}$ provides the expected answer.
\end{ex}

\begin{defn}
The $K$-linear category $\E$ is of finite type if it is Morita equivalent to a finitely presented $K$-algebra. We say that $\E$ is locally of finite type if all of its endomorphism algebras are finitely presented as $K$-algebras.
\end{defn}

\begin{rem}
Note that if $\E$ is of finite type, then the Yoneda embedding
\[\E \longinto \Fun_K(\E^{\op},\Mod_K) \isom \Fun_K(E^{\op},\Mod_K) = \Mod_E\]
maps into the finitely presented projective objects. Thus $\E$ is locally of finite type.
\end{rem}

\begin{thm}\label{NMUprop}
If $\E$ is of finite type, then $\U_{\E}$ is an algebraic stack, locally of finite presentation over $\Spec(K)$. Moreover, if $\E$ is locally proper, the same holds for $\M_{\E}$.
\end{thm}
\begin{proof}
Any Morita equivalence between $\E$ and a finitely presented $K$-algebra $E$ induces an equivalence between $\U_{\E}(S)$ and the groupoid of $(E \otimes_K \O_S)$-modules finitely presented and flat as $\O_S$-modules. It is shown in \cite{TV}, Theorem 1.1, that the resulting forgetful morphism
\[\U_{\E} \longto \coprod_{d\in\NN} \bee\GL_d\]
is representable of finite presentation, implying the result for $\U_{\E}$ by \cite{Stacks}, Lemma \href{http://stacks.math.columbia.edu/tag/05UM}{05UM}.

Now let $\E$ be locally proper. For any $A\in\E$, set $E_A := \End_{\E}(A)$, and $\E_A = \proj_{E_A}$ the corresponding Cauchy completion. Then \eqref{NMU} induces an equivalence
\begin{equation}\label{MvsU}
\M_{\E_A} \isoto \U_{E_A}
\end{equation}
of locally finitely presented algebraic stacks over $K$, by the above.

Since $\E$ is Cauchy complete, the tautological inclusion $E_A \into \E$ induces a representable map $\M_{\E_A} \to \M_{\E}$ by faithfulness, and therefore a representable morphism
\[\rho\!: \coprod_{A\in\pi_0(\E^\times)} \M_{\E_A} \longto \M_{\E}.\]
Moreover, $\rho$ is surjective and smooth, which can be seen locally for each summand. Therefore, if we choose a smooth surjection from a scheme $X_A \to \M_{\E_A}$ for each $A\in\pi_0(\E^\times)$, then the composition
\[\coprod_{A\in\pi_0(\E^\times)} X_A \longto \coprod_{A\in\pi_0(\E^\times)} \M_{\E_A} \longxto{\rho} \M_{\E}\]
will have the same properties. Thus, $\M_{\E}$ is also an algebraic stack (its diagonal being representable by faithfulness), and moreover locally of finite presentation since the source of $\rho$ is. Alternatively, let $(R_i)_{i\in I}$ be a filtered inductive system in $\Aff_K^{\op}$. Then we have the natural equivalence
\[\colim (\E \otimes_K R_i) \isoto \E \otimes_K \colim R_i,\]
since extension of scalars is right exact. Thus, we can conclude again that $\M_{\E}$ is locally finitely presented over $K$, by \cite{Stacks}, Proposition \href{http://stacks.math.columbia.edu/tag/0CMY}{0CMY}.
\end{proof}

For the rest of the section, we consider alternative approaches via reflexive and quasi-coherent families, and compare them to $\M_{\E}$. Since they will not be used hereafter, the reader may safely skip to $\mathsection$\ref{sec7}.

While $\qcoh_S$ and $\refl_S$ are not abelian in general, they are finitely cocomplete, which we will make use of in this section. For $\qcoh_S$, this is most easily seen by using that it is exactly the category of finitely presented objects in the abelian category of quasi-coherent sheaves on $S$. In fact, if $S = \Spec(R)$, then
\[\qcoh_S = \mod_R\]
is the completion of $R$ under finite colimits (as a $K$-linear category). This makes it the natural candidate for our below considerations.

For $\refl_S$ on the other hand, we can make use of the following basic result.

\begin{lemma}[cf. \cite{Hart}, Corollary 1.2]
Let $S$ be a qcqs scheme, $F\in\qcoh_S$. Then $F^{\vee} \in \refl_S$.
\end{lemma}
\begin{proof}
Locally on $S$, we can find a flat cover $E_1 \to E_0 \to F \to 0$, with $E_i\in\vect_S$. Then
\[\begin{tikzcd}
    0 \ar[r] & F^{\vee} \ar[r] \ar[d, "{f}"] & E_0^{\vee} \ar[r] \ar[d, "{\sim}"] & Q \ar[r] \ar[d, "{g}"] & 0 \\
    0 \ar[r] & F^{\vee\vee\vee} \ar[r] & E_0^{\vee\vee\vee} \ar[r] & Q^{\vee\vee} & 
\end{tikzcd}\]
is a commutative diagram of $\O_S$-modules with exact rows, where $Q := \im(E_0^{\vee} \to E_1^{\vee})$. By the snake lemma, $\ker(f)=0$ and $\coker(f) \isom \ker(g) = 0$, since $Q \longinto E_1^{\vee}$ factors over $g$.
\end{proof}

\begin{cor}
The category $\refl_S$ is finitely cocomplete.
\end{cor}
\begin{proof}
The functor $\qcoh_S \longto \refl_S,\ F \longmapsto F^{\vee\vee}$, is left adjoint to the inclusion, via
\[\Hom(F^{\vee\vee},E) \isoto \Hom(F,E),\ f \longmapsto f \circ (F \to F^{\vee\vee}),\ g^{\vee\vee} \longmapsfrom g.\]
Hence the bidual of the colimit taken in $\qcoh_S$ of a finite diagram defines a colimit in $\refl_S$.
\end{proof}

\begin{rem}\label{nocoker}
Let $K$ be a field, and consider the local scheme $S = \Spec(K[[x,y]]/(xy))$, with cyclic module
\[(x) \subseteq \O_S.\]
This is easily seen to be reflexive directly, but of course it cannot be free. In fact, $(x)$ is the cokernel of $y\!: \O_S \to \O_S$ in $\refl_S$, which does not have a cokernel in $\vect_S$, because $y$ would have to act trivially on it. Thus $\vect_S$ is not finitely cocomplete, hence absent from the following definition.
\end{rem}

From now on, we assume that the category $\E$ is finitely cocomplete.

\begin{defn}
The moduli stack of coherent, resp. reflexive, families of objects in $\E$ is the stack $\Q_{\E}$, resp. $\R_{\E}$, over $K$, defined as the \'etale stackification of the functor
\begin{equation}\label{QRdefn}
  \begin{aligned}
  & \Aff_K \longto \Grpd,\ S \longmapsto (\E \btimes_K \qcoh_S)^{\simeq}, \\
  \mbox{resp. } & \Aff_K \longto \Grpd,\ S \longmapsto (\E \btimes_K \refl_S)^{\simeq},
  \end{aligned}
\end{equation}
where $- \btimes_K -$ denotes Kelly's tensor product of finitely cocomplete $K$-linear categories.
\end{defn}

\begin{rem}
Let $f\!: S' \to S$ be any morphism in $\Aff_K$. We obtain a (right exact) functor
\[\E \btimes_K f^*\!:\ \E \btimes_K \qcoh_S \longto \E \btimes_K \qcoh_{S'}\]
from the universal property of the tensor product (\cite{LF}, Theorem 7). Similarly, the functor
\[\refl_S \longto \refl_{S'},\ F \longmapsto (f^*F)^{\vee\vee},\]
is right exact by definition, whence we obtain $\E \btimes_K \refl_S \longto \E \btimes_K \refl_{S'}$.

In general, it is not clear whether the stackification in \eqref{QRdefn} is redundant. However, as a special case of \cite{Sch}, Theorem 5.1, if $\E$ carries a symmetric monoidal structure, right exact in both variables, then $- \btimes_K -$ is a $2$-coproduct, thus commutes with finite $2$-limits.

Therefore, in this case, we are reduced to (fpqc-)descent for quasi-coherent modules, which implies descent for $\qcoh_S$ since finite presentation descends. For $\refl_S$, we use that the dual
\begin{equation}
f^* F^{\vee} = \HHom_{\O_S}(F,\O_S) \otimes_{\O_S} \O_{S'} \isoto \HHom_{\O_{S'}}(F \otimes_{\O_S} \O_{S'}, \O_{S'}) = (f^*F)^{\vee}
\end{equation}
is stable under (faithfully) flat base change $S' \to S$.
\end{rem}

\begin{rem}
If $K$ is a field, then $\Q_{\E}$ has the expected groupoid of $K$-points $\E^{\simeq}$, since
\[\D \isoto \Rex_K(\qcoh_K, \D) = \Rex_K(\vect_K, \D),\ D \longmapsto (K^n \mapsto D^{\oplus n}),\]
for every finitely cocomplete $K$-linear category $\D$. So indeed, $\E$ satisfies the universal property
\[\Rex_K(\E, \D) \isoto \Rex_K(\E, \Rex_K(\qcoh_K, \D))\]
of $\E \btimes_K \qcoh_K$, and $\Q_{\E}(K) = (\E \btimes_K \qcoh_K)^{\simeq}$ by the universal property of the stackification. This argument of course also shows that $\R_{\E}(K) \isom \E^{\simeq}$ as well.

Moreover, if $\E$ is an abelian category of finite length, and $L|K$ is a finite extension, then
\[\E \btimes_K \qcoh_L = \E \btimes_K \vect_L = \E \btimes_K \refl_L\]
coincides with Deligne's tensor product of abelian categories, which follows as a special case from \cite{LF}, Proposition 22, together with \textit{loc.cit.}, Theorem 18.
\end{rem}

\begin{lemma}
There are natural representable morphisms of stacks over $K$, as follows.
\begin{equation}\label{QNR}
\Q_{\E} \xfrom{\phi} \M_{\E} \xto{\rho} \R_{\E}
\end{equation}
\end{lemma}
\begin{proof}
Let $S = \Spec(R) \in \Aff_K$, and let $\D$ be a finitely cocomplete $K$-linear category with symmetric monoidal structure, right exact in both variables, and $\End(\mathbbm 1) = R$. In particular, this applies to $\D \in \{\qcoh_S,\refl_S\}$. The tautological embedding $R \into \D$ induces a natural functor
\[\E \otimes_K R \longinto \E \otimes_K \D \longto \E \btimes_K \D.\]
Thus, we obtain maps $\phi$ and $\rho$ as in \eqref{QNR} by composition with the respective adjunction morphism of the stackification (which is representable). Indeed, the tautological functor from above yields
\[\E \hatotimes_K R \longinto \E \hatotimes_K \D.\]
By construction, both $\E \hatotimes_K \D$ and $\E \btimes_K \D$ are full subcategories of $\Fun_K(\E^{\op} \otimes_K \D^{\op},\Mod_K)$. The former consists of finite direct sums of retracts of representable functors. In particular, these are left exact and finitely presented. But that means that they lie in $\E \btimes_K \D$.

The resulting morphisms $\phi\!: \M_{\E} \to \Q_{\E}$ and $\rho\!: \M_{\E} \to \R_{\E}$ are representable by fully faithfulness.
\end{proof}

\begin{rem}
It is explained in \cite{LF}, Remark 9, when the functor $\E \otimes_K \D \longto \E \btimes_K \D$ is fully faithful. In particular, this applies when $\Spec(K)$ is reduced of dimension $0$.
\end{rem}

\begin{ex}
Let $K$ be a field, and $X$ be a smooth projective (geometrically connected) curve over $K$. Using Proposition \ref{reflXxS} below, we get a similar factorization as in \eqref{Bunfact}, with $\R_{\E}$ in place of $\U_{\E}$. Namely, the functor
\begin{equation}\label{prXprS}
\vect_X \otimes_K \refl_S \longto \refl_{X\times_K S},\ E \otimes F \longmapsto \pr_X^*E \otimes_{\O_{X \times S}} \pr_S^*F,
\end{equation}
is right exact in both arguments, hence yields a functor $\vect_X \btimes_K \refl_S \longto \refl_{X\times_K S}$, which is fully faithful, because \eqref{prXprS} is. Hence the inclusion of $\vect_{X \times_K S}$ induces a map as claimed,
\[\M_{\vect_X} \isoto \Bun_X \xto{\ \pi\ } \R_{\vect_X}.\]
By \cite{Hart}, Corollary 1.4, if $S$ is regular of dimension $\leq 1$, then $\pi$ is an equivalence on $S$-points. On the other hand, consider $S = \Spec(K[[x,y]]/(xy))$, as in Remark \ref{nocoker}. Since $X$ is smooth, in particular $\pr_S$ is flat, and thus $\pr_S^*(x) \in \refl_{X \times_K S}$ is not locally free.

Similarly, we have a factorization (where $X$ can be any proper scheme over $K$)
\[\M_{\qcoh_X} \isoto \Coh_X \xto{\ \psi\ } \Q_{\qcoh_X}\]
by the proof of Proposition \ref{reflXxS}. Then $\psi$ is an equivalence on $S$-points if $S$ is reduced of dimension $0$.
\end{ex}

\begin{prop}\label{reflXxS}
Let $K$ be a field. Let $X$ be a smooth projective (geometrically connected) curve over $K$, and $S \in \Aff_K$. Then the functor \eqref{prXprS} induces an equivalence of categories
\begin{equation}\label{prXboxprS}
\vect_X \btimes_K \refl_S \isoto \refl_{X \times_K S}.
\end{equation}
\end{prop}
\begin{proof}
Since $S$ is affine over $K$, we can write $S$ as $S = \lim S_{\lambda}$, with each $S_{\lambda} \in \Aff_K$ noetherian over $K$, and hence $X \times_K S = \lim(X \times_K S_{\lambda})$. The category $\qcoh_S$ is exhausted by coherent sheaves on the $S_{\lambda}$, in the sense that there is an equivalence of categories
\begin{equation}\label{qcohXxS}
\colim \qcoh_{X \times_K S_{\lambda}} \isoto \qcoh_{X \times_K S}
\end{equation}
by \cite{EGA}, Th\'eor\`eme 8.5.2 $(ii)$. Now, we can apply \cite{Sch}, Theorem 1.7, which implies that
\[\qcoh_X \btimes_K \qcoh_{S_{\lambda}} \isoto \qcoh_{X \times_K S_{\lambda}}\]
is an equivalence, for all $\lambda$. The universal property implies that this commutes with \eqref{qcohXxS}. Therefore, we may assume that $S$ is noetherian, and for $Q \in \refl_{X \times_K S}$, there exist $E_i\in\qcoh_X$ as well as $F_i\in\qcoh_S$ with
\[Q \isom \colim(\pr_X^*E_i \otimes \pr_S^*F_i) = \pr_X^*E \otimes \pr_S^*F, \mbox{ with } E = \colim E_i \mbox{ and } F = \colim F_i.\]
Note that a fortiori, the first $\colim$ is taken in $\refl_{(-)}$, whereas a priori, the other two are not. However, we can assume $E\in\vect_X$ by replacing it by its torsion-free quotient $E/E_{\tors}$, since by right exactness of the functors $\pr_X^*$ and $-\otimes\pr_S^*F$, the sequence
\[\pr_X^*(E_{\tors}) \otimes \pr_S^*F \xto{=0} Q \longto \pr_X^*(E/E_{\tors}) \otimes \pr_S^*F \longto 0\]
is exact, since $Q$ is torsion-free. Then $\pr_S^*F\in\refl_{X \times_K S}$ as well, because it is locally a direct summand of $Q$. Finally, $\pr_S^*F \isom (\pr_S^*F)^{\vee\vee} \isom \pr_S^*(F^{\vee\vee})$, and $F^{\vee\vee}\in\refl_S$. Thus \eqref{prXboxprS} is essentially surjective, and we had already seen that \eqref{prXprS} is fully faithful.
\end{proof}

\begin{rem}[cf. \cite{CG}, Definition 2.1]
Let $\E$ be a cocomplete $K$-linear category. The moduli stack of points of $\E$ is the \'etale stackification $\P_{\E}$ of the functor
\[\Aff_K \longto \Grpd,\ S \longmapsto \operatorname{Pt}_{\E}(S).\]
Note that unlike \cite{CG}, we do not consider the scheme version of $\P_{\E}$. Another possible construction of the moduli stack of objects of $\E$ (not necessarily cocomplete) is then as a substack of $\Q_{\E}$ defined by conditions analogous to $\P_{\E}$. However, while this might be the most accurate general definition, both sufficient functoriality as well as the analogue of Theorem \ref{NMUprop} do not seem easily accessible.
\end{rem}

\section{Unital \texorpdfstring{$2$}{2}-Segal objects and transfer theories}\label{sec7}

In this section, we explain how to extract genuine algebras from $2$-Segal objects, with a particular focus on geometric examples based on $\mathsection$\ref{sec6}. Let $\C$ be an $\infty$-category with finite limits.

\begin{defn}
A unital $2$-Segal object in $\C$ is a simplicial object $X$ such that the maps
\[\begin{gathered}
X_{\{0,\ldots,i,i+2,\ldots,n\}} \longto X_n \times_{X_{\{i,i+1\}}} X_{\{i\}} \\
X_n \longto X_{\{0,1,2\}} \times_{X_{\{0,2\}}} X_{\{0,2,3\}} \times_{\{0,3\}} \cdots \times_{X_{\{0,n-1\}}} X_{\{0,n-1,n\}} \\
X_n \longto X_{\{0,1,n\}} \times_{X_{\{1,n\}}} X_{\{1,2,n\}} \times_{X_{\{2,n\}}} \cdots \times_{X_{\{n-2,n\}}} X_{\{n-2,n-1,n\}}
\end{gathered}\]
are equivalences in $\C$, for every $n \geq 2$ and $0 \leq i < n$.
\end{defn}

\begin{ex}\label{Sdot2Segal}
Let $(\E, \trl{\E}, \trr{\E})$ be a proto-exact category (cf. \cite{higherS}, Definition 4.1). Then the Waldhausen construction $S(\E)$ is unital $2$-Segal (\cite{HSS1}, Example 2.5.4; \cite{GKT}, Theorem 10.10).
\end{ex}

\begin{ex}\label{hecke}
Let $G$ be a group and let $F$ be a $G$-set. By \cite{HSS1}, Proposition 2.6.3, the \v{C}ech nerve $N(G,F)$ of $[G \backslash F] \to \bee G$ is a Segal groupoid, and hence unital $2$-Segal (\textit{loc.cit.}, Proposition 2.5.3). Note that the functor
\[[G \backslash F^{n+1}] \to [G \backslash F] \times_{\bee G} [G \backslash F^n],\ (x,y) \mapsto (x,y,1),\]
is fully faithful, and for $x\in F,\ y\in F^n$ and $h \in G$, there is an isomorphism
\[(h,1)\!: (x,y,h) \isoto (hx,y,1).\]
Thus, by induction, $N(G,F) \isom [G \backslash F^{\bullet +1}]$, with the evident simplicial maps.

If $F$ is a scheme with the action of a group scheme $G$ over a commutative ring $K$, we can form the analogous simplicial stack $\N(G,F) \isom [G \backslash F^{\bullet +1}]$. Again, $\N(G,F)$ is Segal, by Lemma \ref{segalpts} below.
\end{ex}

\begin{lemma}\label{segalpts}
Let $K$ be a commutative ring, and $d\in\NN$. A simplicial stack $\X$ is lower, resp. upper, $d$-Segal if and only if for every $R\in\Aff_K^{\op}$, the groupoid of points $\X(R)$ is lower, resp. upper, $d$-Segal. Similarly, $\X$ is unital $2$-Segal if and only if all $\X(R)$ are unital $2$-Segal.
\end{lemma}
\begin{proof}
This is clear, because limits are computed pointwise.
\end{proof}

Let $\V$ be a symmetric monoidal category. We abbreviate the terminology of \cite{HSS1}, where the following is referred to as a $\V$-valued theory with transfer on $\C$ (\textit{op.cit.}, Definition 8.1.2).

\begin{defn}
Consider two classes of morphisms in $\C$, referred to as $\T$-smooth and $\T$-proper, respectively, closed under composition, and containing all equivalences. Moreover, if
\begin{equation}\label{smoothproper}
  \begin{tikzcd}
    \X' \ar[r, "{f'}"] \ar[d, "{g'}"'] & \Z' \ar[d, "{g}"] \\
    \X \ar[r, "{f}"] & \Z
  \end{tikzcd}
\end{equation}
is a fibre product diagram, where $f$ is $\T$-proper and $g$ is $\T$-smooth, then $f'$ is a $\T$-proper map and $g'$ is $\T$-smooth.

A transfer theory $\T\!: \C \dashto \V$ is an assignment on objects, which extends to a contravariant functor on $\T$-smooth morphisms $g \mapsto g^*$, and a covariant functor $f \mapsto f_*$ on $\T$-proper morphisms. They are called pullback and pushforward, respectively, and satisfy the base change property
\[(f')_*(g')^* = g^*f_*\]
for every diagram as in \eqref{smoothproper}. Finally, for $\Z,\Z'\in\C$, there is an external product
\[\btimes\!: \T(\Z) \otimes \T(\Z') \longto \T(\Z \times \Z'),\]
natural with respect to (pullback along) $\T$-smooth maps, as well as an isomorphism
\[\T(\pt) \isoto 1_{\V},\]
which together satisfy the evident associativity and unitality constraints.
\end{defn}

\begin{defn}
A unital $2$-Segal object of $\C$ is called $\T$-transferable if both of the correspondences
\begin{equation}\label{SdotCorr}
  \begin{tikzcd}
    \Z_2 \ar[d, "{(\del_2,\del_0)}"'] \ar[r, "{\del_1}"] & \Z_1 & \phantom{i} \ar[d, phantom, "{\mbox{and}}"] & \Z_0 \ar[r, "{\sigma_0}"] \ar[d] & \Z_1 \\
    \Z_1 \times \Z_1 & & \phantom{i} & \pt & 
  \end{tikzcd}
\end{equation}
are $\T$-(smooth, proper).
\end{defn}

\begin{prop}[cf. \cite{HSS1}, Proposition 8.1.7]\label{HA}
Let $\T\!: \C \dashto \V$ be a transfer theory, and let $\Z$ be a $\T$-transferable unital $2$-Segal object of $\C$. Then $\H(\Z,\T) := \T(\Z_1)$, endowed with the multiplication map
\[\T(\Z_1) \otimes \T(\Z_1) \xto{\ \boxtimes\ } \T(\Z_1\times\Z_1) \xto{(\del_2,\del_0)^*} \T(\Z_2) \xto{(\del_1)_*} \T(\Z_1),\]
is a unital $\T(\pt)$-algebra in $\V$, which is called the Hall algebra of $\Z$ along $\T$.
\end{prop}

\begin{ex}\label{RHA}
Let $\E$ be a proto-exact category, and suppose that $\#\Ext_{\E}^i(A,B) < \infty$ for $i=0,1$, and for all $A,B\in\E$. The Ringel-Hall algebra of $\E$ is the convolution algebra
\[\Hall(\E) := \H(S(\E)^{\simeq},\QQ[\pi_0(-)])\]
of finitely supported functions on $\pi_0(\E^{\simeq})$ with values in $\QQ$. The product of $\phi,\psi\in\QQ[\pi_0(\E^{\simeq})]$ is defined by
\[(\phi * \psi)(E) = \sum_{B \leq E} \phi(B)\psi(E/B).\]
Here, $\T = \QQ[\pi_0(-)]$ is a transfer theory on $\Grpd$, where a morphism $f$ is $\T$-smooth if $\pi_0(f)$ has finite fibres, and $\T$-proper if the induced map on automorphism groups has finite kernel and cokernel. The external product is given by pointwise multiplication, the pullback is via post-composition, and pushforward via integration over the fibres, that is,
\[\int_{Z} \phi := \sum\limits_{z\in\pi_0(Z)} \frac{\phi(z)}{\#\Aut(z)}, \mbox{ for } \phi\in\QQ[\pi_0(Z)].\]

Applying the transfer theory $\QQ[\pi_0(-)]$ to $N(G,G/K)$ instead (cf. Example \ref{hecke}), where $K \subseteq G$ is an almost normal subgroup, we obtain the double coset Hecke algebra of $(G,K)$.
\end{ex}

\begin{defn}
Let $\Z\in\St K$, and let $\St \Z$ be the category of stacks fibered over $\Z$. Denote by $\Sta \Z$ the full subcategory of $\St \Z$ formed by algebraic stacks of finite presentation and with affine geometric stabilizer groups over $K$.

The Grothendieck group of stacks $K_0(\Sta \Z)$ is defined as the free $\ZZ$-module on geometric equivalence classes of objects of $\Sta \Z$, modulo the following relations.
\begin{equation}\label{K0StRel}
  \begin{gathered}
    [\X \disj \X'] = [\X] + [\X'], \\
    [\X_1] = [\X_2] \mbox{ for all } \X_i \to \X \mbox{ representable by } \\
    \mbox{Zariski locally trivial fibrations of schemes with the same fibres}.
  \end{gathered}
\end{equation}
The maps appearing in the second relation are not to be confused with the (Grothendieck) fibrations we usually consider. We endow $K_0(\Sta K)$ with its ring structure, defined via the fibre product over $K$. As usual, we write $\Sta K$ instead of $\Sta \Spec(K)$.
\end{defn}

\begin{ex}\label{motTr}
A morphism $g\!: \Z' \to \Z$ in $\St K$ is $K_0(\Sta -)$-smooth if $\X \times_{\Z} \Z'$ lies in $\Sta \Z'$ for every $\X\in\Sta \Z$, whereas all morphisms are $K_0(\Sta -)$-proper. Then the pullback $g^*$ is via fibre product, and the pushforward is by pre-composition. The external product is induced by the fibre product over $K$.
\end{ex}

In particular, if $\Z,\Z'\in\St K$ have affine geometric stabilizers, then a morphism $\Z \to \Z'$ of finite presentation is $K_0(\Sta -)$-smooth.

From now on, we assume that the category $\E$ is $K$-linear and that $\bigoplus_{i \in \NN} \Ext^i(A,B) \in\proj_K$ for all $A,B\in\E$, so that in particular, $\E$ is locally proper.

\begin{defn}
We denote by $\S(\E)$ the simplicial stack over $K$ whose $n$-cells are given by the moduli stack of objects in $S_n(\E)$.
\end{defn}

By the above, this ensures that $(\del_2,\del_0)\!: \S_2(\E) \to \S_1(\E) \times_K \S_1(\E)$ is $K_0(\Sta -)$-smooth, by Theorem \ref{NMUprop}, and because the stabilizer group of every $A \in \M_{\E}$ is affine as a closed subscheme of $\Spec(\Sym(\End(A)^{\vee}))^{\times 2}$. Thus, $\S(\E)$ is $K_0(\Sta -)$-transferable.

\begin{defn}\label{motHA}
The resulting $K_0(\Sta K)$-algebra
\[\H(\E) := \H(\S(\E),K_0(\Sta -))\]
is called the motivic Hall algebra of $\E$.
\end{defn}

\begin{ex}
A smaller version of $\H(\E)$ arises by way of a transfer theory $K_0^{\rep}(\Sta -)$ on $\St K$. Namely, $K_0^{\rep}(\Sta \Z)$ is the subgroup of $K_0(\Sta \Z)$ generated by all $[\X \xto{\sigma} \Z]$ such that $\sigma$ is representable. Consequently, $K_0^{\rep}(\Sta -)$-proper maps have to be adjusted to consist of the representable morphisms, while everything else remains the same. Let us write $K_0^{\rep}(\Sta K) =: K_0(\Var K)$.

Since $\del_1$ is faithful, Proposition \ref{HA} applies, and we obtain a $K_0(\Var K)$-algebra
\[\H^{\rep}(\E) := \H(\S(\E),K_0^{\rep}(\Sta -)).\]
However, it is insufficient for our purposes, as $K_0^{\rep}(\Sta -)$ does not satisfy the hypotheses of Theorem \ref{intthm}, the map $\oint$ not being representable.

In \cite{Joy2}, Definition 2.8, the analogues of $K_0^{\rep}(\Sta -) \subseteq K_0(\Sta -)$ are called $\operatorname{SF} \subseteq \ubar{\operatorname{SF}}$.
\end{ex}

\begin{rem}
The notation $K_0(\Var K)$ is justified because it is generated by classes of integral affine schemes of finite presentation over $K$. In short, this follows from the existence of the schematic locus and the fact that $X_{\operatorname{red}} \to X$ as well as $U \disj (X \minus U) \to X$ are geometric equivalences, if $U \subseteq X$ is an open subscheme (cf. \cite{Br}, Lemma 2.12).

In particular, if $K$ is a field, this is the usual Grothendieck ring of varieties over $K$. The third relation in \eqref{K0StRel} is redundant here (cf. \textit{loc.cit.}, Lemma 2.5, which is valid over arbitrary fields). If $\cha K = 0$, by \textit{loc.cit.}, Lemma 2.9, we can equivalently consider the free $\ZZ$-module on \textit{isomorphism} classes of varieties over $K$, subject to the single relation
\[[X] = [Z] + [X \minus Z] \mbox{ for every closed subvariety } Z \subseteq X.\]
For any field $K$, any $\X\in\Sta K$ is geometrically equivalent to a finite disjoint union of quotient stacks of the form $[X/\GL_N]$ with $X$ an algebraic space over $K$ by \cite{Kr}, Proposition 3.5.9 and Proposition 3.5.6. But $[X] = [\GL_N][X/\GL_N]$, and
\[[\GL_N] = (\LL^N-1)\cdots(\LL^N-\LL^{N-1}) = \prod_{n=1}^N \LL^{N-n}(\LL^n-1) = \LL^{\frac{N(N-1)}{2}}\prod_{n=1}^N (\LL^n-1),\]
where $\LL := [\AA^1_K]$ denotes the Lefschetz motivic class. This implies that the map of rings
\begin{equation}\label{K0VarK0St}
K_0(\Var K)[\LL^{-1},(\LL^n-1)^{-1} \mid n\in\NN] \isoto K_0(\Sta K)
\end{equation}
is an isomorphism (also see \cite{Br}, Lemma 3.9, and \cite{Toen}, Th\'eor\`eme 3.10 \textit{ff}). Note that it is shown in \cite{Bori} that $\LL$ is a zero-divisor, at least if $\cha K = 0$.
\end{rem}

Motivic classes are interesting because they are the universal motivic invariants. Namely, a (multiplicative) motivic measure with values in an abelian group (resp. ring) $Q$ is defined to be a morphism $\mu\!: K_0(\Var K) \to Q$.

\begin{ex}\label{motMeasEX}
\begin{enumerate}[label=$(\arabic*)$]
\item Suppose $K$ is a finite field. The motivic counting measure is the map
\[\mu_{\#}\!: K_0(\Var K) \longto \ZZ,\ [X] \longmapsto \#X(K).\]
Moreover, this extends to the Grothendieck ring of stacks as
\[\mu_{\#}\!: K_0(\Sta K) \longto \QQ,\ [\X] \longmapsto \#\X(K),\]
that is, it commutes with \eqref{K0VarK0St}. For $L|K$ finite, we also consider $\mu_{\#} \otimes_K L\!: [\X] \mapsto \#\X(L)$.

\item\label{measureb} Let $G_K = \Gal(K^{\sep}|K)$, and $\l \neq \cha K$ prime. The virtual Euler-Poincar\'e characteristic
\[\chi_c\!: K_0(\Var K) \longto K_0(\rep_{\Qlbar}\!(G_K)),\ X \longmapsto \sum_{i\in\ZZ} (-1)^iH_c^i(X,\Qlbar),\]
is a motivic measure. Here, $H_c^i(X,\Qlbar)$ denotes the $\l$-adic \'etale cohomology of $X \otimes_K \bar K$ with compact support. Composing $\chi_c$ with $K_0(\rep_{\Qlbar}\!(G_K)) \xto{\dim} K_0(\vect_{\Qlbar}) \cong \ZZ$ yields the usual Euler characteristic (if $K=\CC$, with respect to Betti cohomology).

Using \eqref{K0VarK0St}, we can use this to obtain a map on the Grothendieck ring of stacks,
\[\chi_c\!: K_0(\Sta K) \longto K_0(\rep_{\Qlbar}\!(G_K))[(t^{-n}-1)^{-1} \mid n\in\NN],\]
where $t^{-1} := \Qlbar(-1) = (\lim \mu_{\l^m}(K^{\sep}) \otimes_{\ZZ_{\l}} \Qlbar)^{\vee}$ denotes the inverse Tate twist.

\item The name motivic measure comes from the following instance. Assume that $K$ has resolutions of singularities, for example $\cha K = 0$. Denote by $\CHMeff K \subseteq \CHM K$ the category of (effective) Chow motives over $K$. Then the functor
\[h\!: \SmVar K \longto \CHMeff K,\ X \longmapsto (X,\Delta_X),\]
descends to a morphism on the corresponding Grothendieck rings
\begin{equation}\label{chowMot}
h\!: K_0(\Var K) \longto K_0(\CHMeff K).
\end{equation}
As in \ref{measureb}, we moreover extend this to stacks as
\[h\!: K_0(\Sta K) \longto K_0(\CHM K)[([L]^n-1)^{-1} \mid n\in\NN],\]
where $L$ is the Lefschetz motive. Similarly, for arbitrary $K$, we obtain a measure
\[K_0(\Var K) \longto K_0(\operatorname{DM_c}(K))\]
with values in the Grothendieck ring of Voevodsky's category.

\item Suppose $K = \CC$. Then Hodge-Deligne polynomials are another example of a motivic measure. Moreover, Larsen-Lunts \cite{LL} showed that $K_0(\Var\CC)/(\LL) \isom \ZZ[\operatorname{SB}]$ is the monoid ring on stable birational equivalence classes.

\item The Albanese variety yields a motivic measure valued in the semigroup ring of isogeny classes of abelian varieties; this and further examples are listed in \cite{MotMeas}.

\item A more exotic motivic measure has been introduced in \cite{AMotMeas}, where the value ring is given by the connected components of Waldhausen's algebraic $K$-theory of spaces.

\item Assume that $K$ has resolutions of singularities. Then we obtain a motivic measure with values in the Grothendieck ring of smooth proper pretriangulated dg-categories as in \cite{BLL}, which by \cite{Tab} agrees with To\"en's secondary $K$-theory \cite{dg} and carries a non-commutative analogue of \eqref{chowMot}, $K_0(\Var K) \xto{\Perf} K^{(2)}_0(K) \xto{U(-)} K_0(\operatorname{NChow}(K))$.
\end{enumerate}
\end{ex}

\begin{ex}
We will make use of another variant of the above, the analytic motivic Hall algebra over a non-archimedean field $K$. It is given by the transfer $\T = K_0(\AnSta (-)^{\an})$ as defined below, where $(-)^{\an}\!: \St K \to \AnSt K$ denotes the analytification functor from \cite{Ul}.

Here, we work with Berkovich spaces in order to conform with \textit{op.cit.} as well as \cite{DOR}. If $K$ is non-trivially valued, an alternative approach based on \cite{PY} and \cite{RZ} within the framework of rigid-analytic geometry should also be possible, and similarly for Huber's theory of adic spaces.
\end{ex}

\begin{defn}
For a stack $\Z \in \AnSt K$, the abelian group $K_0(\AnSta \Z)$ is defined as follows. Let $\AnSta \Z$ be the full subcategory of $\AnSt \Z$ on all compact analytic stacks over $K$ whose geometric stabilizers are closed subgroups of $\GL_N^{\an}$ for some $N \gg 0$.

A morphism $\X_0 \to \X$ in $\AnSta \Z$ is called a geometric equivalence if, for every algebraically closed non-archimedean field $C|K$, the map $\X_0(C) \to \X(C)$ is an equivalence.

The Grothendieck ring of analytic stacks $K_0(\AnSta \Z)$ is the free $\ZZ$-module on geometric equivalence classes of objects of $\AnSta \Z$, modulo the relations \eqref{K0StRel}, where instead of Zariski locally trivial fibrations of course we consider the analytic topology.
\end{defn}

\begin{rem}
The analytification functor induces a map
\[(-)^{\an}\!: K_0(\Sta K) \isom K_0(\Var K)[[\GL_N]^{-1} \mid N \in \NN] \longto K_0(\AnSta K),\]
which assigns to $[\X] \in K_0(\Sta K)$ with $\X = [X/\GL_N]$ with $X$ proper over $K$ the class
\[[\X]^{\an} := [[X/\GL_N]^{\an}] = [X^{\an}/\GL_N^{\an}] \in K_0(\AnSta K).\]
This is well-defined since $|\X^{\an}| = |X^{\an}|/|\GL_N^{\an}|$ is compact since $X$ is proper.
\end{rem}

We denote the $K_0(\AnSta K)$-algebra we obtain from this by $\H^{\an}(\E) = K_0(\AnSta \M_{\E}^{\an})$.

\begin{rem}
As in the algebraic setting, it should be true that there is an isomorphism
\[K_0(\AnVar K)[[\GL_N^{\an}]^{-1} \mid N\in\NN] \isoto K_0(\AnSta K),\]
where $K_0(\AnVar K) := K_0^{\rep}(\AnSta K)$ is defined similarly to its algebraic version. Indeed, the same proof should work, based on \cite{Kr}, with the main subtlety given by the less well-behaved notion of generic flatness in the analytic world, as explained in \cite{Ducros}, $\mathsection$10.3.
\end{rem}

\begin{defn}
Let $\T\!: \C \dashto \V$ and $\T'\!: \C' \dashto \V'$ be transfer theories. A morphism $\Phi\!: \T \to \T'$ consists of a right adjoint $\C \to \C'$ sending $\T$-smooth, resp. $\T$-proper, maps to $\T'$-smooth, resp. $\T'$-proper, morphisms, as well as a symmetric monoidal functor $\V \to \V'$, together with a transformation between the two compositions
\[\begin{tikzcd}
    \C \ar[r, "{\Phi}"] \ar[d, shift right, "{\T}"'] \ar[d, shift left] & \C' \ar[d, shift right] \ar[d, shift left, "{\T'}"] \\
	\V \ar[r, "{\Phi}"'] \ar[ur, Rightarrow, shorten <= 7pt, shorten >= 7pt] & \V',
\end{tikzcd}\]
which is natural for both pullback and pushforward (as indicated by the double arrows).
\end{defn}

\begin{lemma}\label{countMeas}
Let $K$ be a finite field. The motivic counting measure induces an epimorphism of transfer theories
\[\bbmu_{\#}\!: K_0(\Sta -) \longto \QQ[\pi_0(-)],\]
with respect to the functor $\St K \to \Grpd,\ \Z \mapsto \Z(K)$, with the adjustment that $f\!: \Z \to \Z'$ is only $K_0(\Sta -)$-proper if it is $K_0(\Sta -)$-smooth. Namely, the maps
\begin{equation}\label{countMeasEpi}
\bbmu_{\#}\!: {}_{\ZZ}K_0(\Sta \Z)_{\QQ} \longto \QQ[\pi_0(\Z(K))],\ [\X] \longmapsto (z \mapsto \mu_{\#}(\X \times_{\Z,z} \Spec(K))),
\end{equation}
are $\QQ$-vector space epimorphisms, where ${}_{\ZZ}(-)$ and $(-)_{\QQ}$ denote restriction and extension of scalars, respectively. In fact, \eqref{countMeasEpi} is already surjective when restricted to ${}_{\ZZ}K_0(\Sta \Z)$.
\end{lemma}
\begin{proof}
First of all, $\Z \mapsto \Z(K)$ is a right adjoint functor, since mapping a groupoid $Z$ to the constant stack $\ubar Z$ is a left adjoint, and fibrations of stacks are defined pointwise.

Naturality with respect to pullback is clear. For $K_0(\Sta -)$-proper $f\!: \Z \to \Z'$, the diagram
\[\begin{tikzcd}
    K_0(\Sta \Z) \ar[r, "{f_*}"] \ar[d] \mathclap{\smash{\mbox{\hspace{-5.7em}\raisebox{-3.7ex}{\scalebox{0.7}{\ensuremath{\bbmu_{\#}}}}}}} & K_0(\Sta \Z') \ar[d] \mathclap{\smash{\mbox{\hspace{-3.1em}\raisebox{-3.7ex}{\scalebox{0.7}{\ensuremath{\bbmu_{\#}}}}}}} \\
	\QQ[\pi_0(\Z(K))] \ar[r, "{f(K)_*}"] & \QQ[\pi_0(\Z'(K))]
\end{tikzcd}\]
commutes, since for each $[\X \xto{\sigma} \Z] \in K_0(\Sta \Z)$, we can use functoriality of the pushforward,
\[f(K)_*(\bbmu_{\#}([\X])) = f(K)_*(\sigma(K)_*\mathbbm 1_{\X(K)}) = (f(K) \circ \sigma(K))_*\mathbbm 1_{\X(K)} = \bbmu_{\#}(f_*[\X]).\]

For surjectivity, the element $[\Spec(K) \xto{z} \Z] \in {}_{\ZZ}K_0(\Sta \Z)_{\QQ}$ maps to $\mathbbm 1_z \in \QQ[\pi_0(\Z(K))]$. It suffices to restrict to the integral part ${}_{\ZZ}K_0(\Sta \Z)$, since more generally, for all $n\in\NN$, the function $\frac 1n \cdot \mathbbm 1_z$ has preimage $[\bee\ubar{\ZZ/n} \to \Spec(K) \xto{z} \Z] \in {}_{\ZZ}K_0(\Sta \Z)$.
\end{proof}

\begin{rem}\label{mutimesL}
The adjustment of $K_0(\Sta -)$-proper maps in Lemma \ref{countMeas} is so that $\bbmu_{\#}$ respects them. It does not affect the transfer theory properties of $K_0(\Sta -)$, nor the existence of the Hall algebra, since $\del_1$ is $K_0(\Sta -)$-smooth. Alternatively, we could also simply pull back the $\QQ[\pi_0(-)]$-proper maps to $\St K$.

Let $L|K$ be finite. Then we can replace $\mu_{\#}$ by
\[\mu_{\#} \otimes_K L = \mu_{\#} \circ (- \times_{\Spec(K)} \Spec(L)),\]
to obtain another morphism of transfer theories. Indeed, the base change functor $\Sta K \to \Sta L$ on the right-hand side is right adjoint to the forgetful functor.
\end{rem}

\begin{ex}\label{HallToHall}
In the situation of Example \ref{RHA}, consider the transfer theory
\[\Fun^f(-,\vect_{\CC})\!: \Grpd \dashto \AbCat\]
from \cite{HSS1}, $\mathsection$8.3.C, and \cite{Toby}, $\mathsection$2.5. Here, $\Fun^f$ denotes finitely supported (meaning, on $\pi_0(-)$) functors, and $\AbCat$ is the category of abelian categories with exact functors as morphisms. Then the pullback, resp. pushforward, are defined by post-composition, resp. via left Kan extension. There is an epimorphism of transfer theories
\[\Phi\!: \Fun^f(-,\vect_{\CC}) \longonto \QQ[\pi_0(-)],\]
which is the identity on $\Grpd$, and $\Phi = K_0(-)_{\QQ} = K_0(-) \otimes_{\ZZ} \QQ\!: \AbCat \longto \Vect_{\QQ}$. Namely,
\[\Phi_Z\!: K_0(\Fun^f(Z,\vect_{\CC}))_{\QQ} \longonto \QQ[\pi_0(Z)],\ F \longmapsto (z \mapsto \dim_{\CC}(F(z))).\]
In particular, we obtain a surjective morphism of algebras
\[\Hall^{\otimes}(\E) \longonto \Hall(\E),\]
where $\Hall^{\otimes}(\E) = K_0(\H(S(\E),\Fun^f(-,\vect_{\CC})))$ denotes the Hall monoidal algebra.
\end{ex}

\begin{lemma}\label{morphofTr}
If $\Phi\!: \T \to \T'$ is a morphism of transfer theories, then the assignment
\[\Phi^*\T'\!: \C \longto \V',\ \Z \mapsto \T'(\Phi(\Z)),\]
is a transfer theory, where $f$ is $(\Phi^*\T')$-smooth, resp. $(\Phi^*\T')$-proper, if $\Phi(f)$ is $\T'$-smooth, resp. $\T'$-proper. The external product is given by
\[\T'(\Phi(\Z)) \otimes \T'(\Phi(\Z')) \xto{\ \boxtimes\ } \T'(\Phi(\Z)\times\Phi(\Z')) \isom \T'(\Phi(\Z\times\Z')),\]
while pullback $\Phi(-)^*$ and pushforward $\Phi(-)_*$ are induced by $\T'$.
\end{lemma}
\begin{proof}
We only need to show that the base change property is satisfied. But since $\Phi$ is a morphism of transfer theories, it maps any diagram as in \eqref{smoothproper} in $\C$ to a diagram with the same properties in $\C'$.
\end{proof}

\begin{rem}\label{motRHA}
It follows immediately from the definitions that if $\Phi\!: \T \to \T'$ is a morphism of transfer theories, then for any unital $2$-Segal object $\Z$ of $\C$, the resulting two Hall algebras are the same,
\[\H(\Phi(\Z), \T') \isom \H(\Z, \Phi^*\T').\]
In particular, note that $\Phi(\Z)$ is indeed unital $2$-Segal again, because $\Phi$ is a right adjoint. In the special case from Lemma \ref{countMeas}, we can describe the Ringel-Hall algebra (Example \ref{RHA}) as
\[\Hall(\E) \isom \H(\S(\E),\bbmu_{\#}^*\QQ[\pi_0(-)]).\]
This is useful in that it allows us to essentially work exclusively with $\C = \St K$.
\end{rem}

\section{Equivariant motivic Hall algebras}\label{sec8}

For definitions and some general theory of group actions on stacks, see \cite{Romagny}.

We begin by defining the coefficient ring, which can be interpreted as a Hall algebra in and of itself. Indeed, a unital $2$-Segal groupoid (or category) $X$ with $X_0 = \pt$ defines a monoidal structure on $\St \ubar X_1$ in the usual way (and similarly for $\St (\ubar X_1 \times_K \Z)$). Now, if $X$ is $\QQ[\pi_0(-)]$-transferable, this restricts to a transfer theory $\Sta -$ on $\Cat$, from which we can pass to $K_0(\Sta -)\!: \Cat \dashto \Mod_{K_0(\Sta K)}$.

Moreover, assuming $X_2 \xto{\del_1} X_1$ is faithful, we merely need to require $X_2 \xto{(\del_2,\del_0)} X_1 \times X_1$ have $\pi_0$-finite fibres to obtain an algebra structure on $K_0(\Sta \ubar X_1)$ from the above.

However, note that the transfer theory itself does not restrict to $\Sta -$ in this case. This is analogous to the fact that the Ringel-Hall algebra can be defined with $\ZZ$-coefficients while the transfer theory $\QQ[\pi_0(-)]$ cannot. Since we are only interested in the algebra itself, we will not consider this further.

\begin{defn}
Let $\D$ be a proto-exact category. Then $\D$ is called split proto-exact, if for every pair of objects $M,N \in \D$, the set $\Ext^1_{\D}(M,N)$ of Yoneda extensions has precisely one element.
\end{defn}

\begin{ex}
\begin{enumerate}[label=$(\arabic*)$]
\item Any Segal groupoid $X$ with $X_0 = \pt$ is allowed, since the condition on $\del_1$ is unnecessary in that case.
\item If $\D$ is finitary, of course $S(\D)$ is even $\QQ[\pi_0(-)]$-transferable. But we may drop the finiteness assumption if $\D$ is a split proto-exact category, for example $\D = \Vect_{\FF_1}$ or any split exact category. In particular, any semisimple category is allowed.
\end{enumerate}
\end{ex}

\begin{rem}
This is enough for the main examples we have in mind, in particular, the categories $\D = \vect_{\FF_q}$, $q \geq 1$, as well as the semisimple category of split semisimple isocrystals over the $p$-adic field $K$. As explained in (the proof of) \cite{DOR}, Proposition 8.2.1, the requisite base change properties hold, and every isocrystal is split semisimple over some field extension. With some work, this can be used to descend some results to more general cases.

In general, we may consider the socle of $\D$, that is, the full subcategory $\D^{\ssi} \subseteq \D$ of semisimple objects. Then $K_0(\D) \isom K_0(\D^{\ssi})$, cf. \cite{An}, Corollary 1.4.12.
\end{rem}

\begin{defn}
Let $\Z \in \St K$ and let $\D$ be a category. The $\D$-equivariant Grothendieck group of stacks relative $\Z$ is the free abelian group $K_0^{\D}(\Sta \Z)$ generated by geometric equivalence classes of $\X \in \St (\ubar{\D} \times_K \Z)$ such that $\X \times_{\ubar{\D}} \pi_0(\ubar{\D}) \in \Sta \Z$, modulo the relations
\[\begin{gathered}
  [\X \disj \X'] = [\X] + [\X'], \\
  [\X_1] = [\X_2] \mbox{ for all } \X_i \to \X \mbox{ representable by } \\
  \mbox{Zariski locally trivial fibrations of schemes with the same fibres},
\end{gathered}\]
where $\X_i \to \X$ is understood to be a morphism in $\St (\ubar{\D} \times_K \Z)$, that is, $\D$-equivariant. As before, we write $K_0^{\D}(\Var K) := K_0^{\D, \rep}(\Sta K)$. Similarly, if $K$ is a non-archimedean field, we define $K_0^{\D}(\AnSta \Z)$ and $K_0^{\D}(\AnVar K)$ in the evident manner.
\end{defn}

Just like in the non-equivariant setting, equivariant motivic classes are interesting, because they are the universal equivariant motivic measures (which are more refined invariants than their non-equivariant counterparts).

\begin{ex}\label{eqmotmeas}
\begin{enumerate}[label=$(\arabic*)$]
\item Let $G_K = \Gal(K^{\sep}|K)$, let $\l \neq \cha K$ be a prime, and let $H$ be a group. The $H$-equivariant virtual Euler-Poincar\'e characteristic
\[\chi_c\!: K_0^H(\Var K) \longto K_0(\rep_{\Qlbar}\!(H \times G_K)),\ X \longmapsto \sum_{i\in\ZZ} (-1)^iH_c^i(X,\Qlbar),\]
is an equivariant motivic measure. Similarly, if $H$ is an analytic group, we get
\[\chi_c\!: K_0^H(\AnVar K) \longto K_0(\rep_{\Qlbar}\!(H(K) \times G_K)),\]
where $\rep$ in this context means smooth representations.

\item The eponymous motivic measure has its equivariant analogue in the map
\[K_0^H(\Var K) \longto K_0(\Fun(\bee H, \CHM K)).\]
\item Over $\CC$, both the Hodge-Deligne and Poincar\'e polynomials extend to equivariant motivic measures (cf. \cite{MotMeas}, also for further examples, as well as \cite{EqMotMeas}).
\end{enumerate}
\end{ex}

Again we may extend all of the above to stacks in the usual manner via localization.

\begin{defn}
The $\D$-equivariant motivic transfer theory $K_0^{\D}(\Sta -)$ on $\St K$ has pullback and pushforward defined in the same way as $K_0(\Sta -)$; given $\Z,\Z'\in\St K$, the external product is the composition of the fibre product morphism
\[K_0^{\D}(\Sta \Z) \otimes K_0^{\D}(\Sta \Z') \longto K_0^{\D \times \D}(\Sta (\Z \times_K \Z'))\]
with pullback-pushforward along the correspondence of constant stacks
\begin{equation}\label{SdotCcorr}
  \begin{tikzcd}
    \ubar{S_2(\D)} \ar[r, "{\del_1}"] \ar[d, "{(\del_2,\del_0)}"'] & \ubar{\D} \\
    \ubar{\D} \times_K \ubar{\D} & 
  \end{tikzcd}
\end{equation}
or more precisely, its fibre product with $\Z \times_K \Z'$ over $K$.
\end{defn}

\begin{rem}\label{eqTrId}
Let $\Z$ be a unital $2$-Segal stack. Then multiplication on
\begin{equation}\label{eqTrf}
K_0^{\D}(\Sta \Z_1) \isom K_0(\Fun^f(\D,\Sta \Z_1))
\end{equation}
is defined as in Example \ref{HallToHall}, but with $\vect_{\CC}$ replaced by $\Sta \Z_1$, endowed with the monoidal structure defined by $\Z$. Namely, by pullback, resp. pushforward, of functors $\Fun^f(-,\Sta \Z_1)$ along \eqref{SdotCcorr}, then taking the Grothendieck group. In particular, $K_0^{\D}(\Sta K)$ can be understood as a motivic version of Green's Hall algebra (\cite{Green1}, \cite{Zel}, and \cite{HSS1}, Example 8.3.18), as alluded to above.

Moreover, since the generators of $K_0^{\D}(\Sta \Z_1)$ are qcqs, we have
\[K_0^{\D}(\Sta \Z_1) \isom \!\bigoplus_{N\in\pi_0(\D^{\simeq})}\! K_0^{\Aut(N)}(\Sta \Z_1),\]
where $K_0^{G}(\Sta \Z_1)$ is defined exactly like $K_0^{\D}(\Sta \Z_1)$, but with $\bee G$ in place of $\D$.

%
\end{rem}

Consequently, the following result allows us to extract a new Hall algebra.

\begin{lemma}
The assignment $K_0^{\D}(\Sta -)$ defines a transfer theory on $\St K$, with the same smooth and proper morphisms as $K_0(\Sta -)$.
\end{lemma}
\begin{proof}
This follows from Lemma \ref{eqTrGen} below, in accord with the transfer theory structure from Example \ref{motTr}. Indeed, for a $K_0(\Sta -)$-smooth map $\Z' \to \Z$, and all stacks $\X$ over $\ubar{\D} \times_K \Z$ such that $\X \times_{\ubar{\D}} \pi_0(\ubar{\D}) \in \Sta \Z$, the fibre product
\[\X \times_{(\ubar{\D} \times_K \Z)} (\ubar{\D} \times_K \Z') \times_{\ubar{\D}} \pi_0(\ubar{\D}) = (\X \times_{\ubar{\D}} \pi_0(\ubar{\D})) \times_{\Z} \Z'\]
lies in $\Sta \Z'$.
\end{proof}

\begin{lemma}\label{eqTrGen}
Let $\T\!: \C \dashto \V$ be a transfer theory, and let $\Z$ be a unital $2$-Segal object of $\C$, satisfying the hypotheses of Proposition \ref{HA}. Then the assignment
\[\T^{\Z_1}(-)\!: \C \longto \V,\ \Z \longmapsto \T(\Z_1\times \Z)\]
is a transfer theory with respect to $\T$-smooth and $\T$-proper morphisms, where pullback and pushforward are defined by $\T$ in the apparent way, and external product given as follows,
\[\begin{tikzcd}[column sep=1.7em]
  \T^{\Z_1}(\Z) \otimes \T^{\Z_1}(\Z') \ar[r, "{\boxtimes}"] & \T(\Z_1\times\Z_1\times\Z\times\Z') \ar[r, "{(\del_2,\del_0)^*\!}"] & \T(\Z_2\times\Z\times\Z') \ar[r, "{(\del_1)_*}"] & \T^{\Z_1}(\Z\times\Z')
\end{tikzcd}\]
where the appropriate identity maps are suppressed from the notation.
\end{lemma}
\begin{proof}
Any properties not concerning the external product are essentially tautological. Now, the associativity of the external product is ensured by $\Z$ being $2$-Segal (cf. Proposition \ref{HA}). Let $f\!: \X \to \Z$ be a $\T$-smooth morphism. Then the following diagram commutes.
\[\begin{tikzcd}[column sep=1.7em]
  \T^{\Z_1}(\X) \otimes \T^{\Z_1}(\Z') \ar[dr, phantom, "{\!\!\!(1)}"] \ar[r, "{\boxtimes}"] & \T(\Z_1\times\Z_1\times\X\times\Z') \ar[dr, phantom, "{(2)\!\!\!\!}"] \ar[r, "{(\del_2,\del_0)^*\!}"] & \T(\Z_2\times\X\times\Z') \ar[dr, phantom, "{(3)\!\!\!}"] \ar[r, "{(\del_1)_*}"] & \T^{\Z_1}(\X\times\Z') \\
  \T^{\Z_1}(\Z) \otimes \T^{\Z_1}(\Z') \ar[r, "{\boxtimes}"'] \ar[u, "{f^*}"'] & \T(\Z_1\times\Z_1\times\Z\times\Z') \ar[r, "{(\del_2,\del_0)^*\!}"'] \ar[u, "{f^*}"'] & \T(\Z_2\times\Z\times\Z') \ar[r, "{(\del_1)_*}"'] \ar[u, "{f^*}"'] & \T^{\Z_1}(\Z\times\Z') \ar[u, "{f^*}"']
\end{tikzcd}\]
Namely, square (1) by functoriality of $\btimes$, square (2) is clear, and (3) by base change.
\end{proof}

\begin{defn}\label{eqHA}
Let $\E$ be an exact, $K$-linear category, and $\D$ split proto-exact. The $\D$-equivariant motivic Hall algebra $\H^{\D}(\E)$ of $\E$ is the Hall algebra of $\S(\E)$ with respect to the transfer theory $K_0^{\D}(\Sta -)$,
\[\H^{\D}(\E) := \H(\S(\E),K_0^{\D}(\Sta -)).\]
Similarly, we define the analytic $\D$-equivariant motivic Hall algebra
\[\H^{\D, \an}(\E) := \H(\S(\E),K_0^{\D}(\AnSta (-)^{\an})).\]
\end{defn}

In what follows, for simplicity, we will usually make statements only for $\H^{\D}(\E)$, but they will be equally valid for $\H^{\D, \an}(\E)$.

\begin{rem}\label{eqTrExpl}
The $\D$-equivariant motivic Hall algebra $\H^{\D}(\E)$ carries a $K_0^{\D}(\Sta K)$-module structure. In fact, $K_0^{\D}(\Sta -)$ is a transfer theory with values in the category of $K_0^{\D}(\Sta K)$-modules by necessity. However, of course we may restrict scalars to $K_0(\Sta K)$ or $\ZZ$ as we please.

It is evident from the construction that $\H^{\D}(\E)$ is $K_0(\D) \times K_0(\E)$-graded, where the $K_0(\E)$-grading is defined exactly as for the usual (motivic) Hall algebra (cf. \eqref{K0grading}). In fact,
\[\H^{\D}(\E) \isom \H(\E) \oplus \bigoplus_{0 \neq N \in \pi_0(\D^{\simeq})} K_0^{\Aut(N)}(\Sta \M_{\E}),\]
by Remark \ref{eqTrId}, where the elements $[\X] \in \H(\E)$ carry the trivial action
\[\Z' \longto \pi_0(\ubar{\D}) \longinto \ubar{\D}.\]
Then the external product for $K_0^{\D}(\Sta -)$ is more explicitly described by parabolic induction, as follows. The fibre product morphism on
\[\begin{aligned}
  \!\bigoplus_{N\in\pi_0(\D^{\simeq})}\! K_0(\Sta \bee\ubar{\Aut(N)} \times_K \Z)\ \otimes \!\bigoplus_{M\in\pi_0(\D^{\simeq})}\! K_0(\Sta \bee\ubar{\Aut(M)} \times_K \Z') \\
  = \!\!\bigoplus_{N,M\in\pi_0(\D^{\simeq})}\!\! \big(K_0(\Sta \bee\ubar{\Aut(N)} \times_K \Z) \otimes K_0(\Sta \bee\ubar{\Aut(M)} \times_K \Z')\big)
\end{aligned}\]
is simply given by the component-wise fibre product map with values in
\[\!\!\bigoplus_{N,M\in\pi_0(\D^{\simeq})}\!\! K_0(\Sta \bee\ubar{(\Aut(N)\times\Aut(M))} \times_K (\Z \times_K \Z')).\]
Then pullback-pushforward along \eqref{SdotCcorr} $\times_K\ (\Z \times_K \Z')$ restricts to each component
\[\begin{tikzcd}[column sep=1.7em]
  K_0(\Sta \bee \ubar{P_{N,M}} \times_K \Z \times_K \Z') \ar[r, "{(\del_1)_*}"] & K_0(\Sta \bee\ubar{\Aut(N \oplus M)} \times_K \Z \times_K \Z') \\
  K_0(\Sta \bee\ubar{(\Aut(N)\times\Aut(M))} \times_K \Z \times_K \Z'), \ar[u, "{(\del_2,\del_0)^*}"] & 
\end{tikzcd}\]
where $P_{N,M} = \Aut(N \into N \oplus M \onto M)$ is the automorphism group of the unique extension up to equivalence of $M$ by $N$ (by abuse of notation). That is, we restrict the action along the inclusion
\[P_{N,M} \longinto \Aut(N) \times \Aut(M),\]
and then push forward along $\del_1$, which is given by the parabolic induction map
\begin{equation}\label{parabInd}
\Ind_{P_{N,M}}^{\Aut(N \oplus M)}\!:\ K_0^{P_{N,M}}(\Sta \Z \times_K \Z') \longto K_0^{\Aut(N \oplus M)}(\Sta \Z \times_K \Z').
\end{equation}
\end{rem}

\section{Abstract integration maps}\label{sec9}

\begin{defn}[see (\cite{Schn})]
A quasi-abelian category is an additive category with kernels and cokernels, which are preserved by pushouts and pullbacks, respectively.
\end{defn}

Let $\E, \B$ be $K$-linear quasi-abelian categories, consider a proto-exact category $\D$, together with an exact functor $\nu\!: \D \to \B$, and let $\omega\!: \E \to \B$ be an exact $K$-linear isofibration.

\begin{defn}
If $\Lambda$ is a totally ordered set, we denote by $\Fil{\Lambda}{\B}$ the category of objects with a separated exhaustive decreasing filtration, that is, pairs $(V,F)$ with $V \in \B$ and $F\!: \Lambda^{\op} \to \trl{\B}$ with $\lim F = 0$ and $\colim F = V$. When $\B = \vect_K$, we write $\Fil{\Lambda}{\B} = \Fil{\Lambda}{K}$, and for $L|K$, we set
\[\Fil{\Lambda}{L|K} = \Fil{\Lambda}{L} \times_{\vect_L} \vect_K.\]
\end{defn}

\begin{ex}\label{isofibEX}
Let $\B$ be a quasi-abelian category.
\begin{enumerate}[label=$(\arabic*)$]
\item\label{isofiba} For a totally ordered set $\Lambda$, the fibre functor $\omega\!: \Fil{\Lambda}{\B} \longto \B$ is an isofibration. Given a filtered object $(V,F^\bullet)\in\Fil{\Lambda}{\B}$ and an isomorphism $g\!: V \isoto W$ in $\B$, we can lift it to
\[g\!: (V,F^\bullet) \isoto (W,gF^\bullet) \mbox{ in } \Fil{\Lambda}{\B}.\]

\item Consider the fibre functor $\Rep_{\B}(Q) \to \B$ for a quiver $Q$. Let $g$ be an isomorphism
\[g\!: \bigoplus_{i \in Q_0} A_i \isoto W,\]
with $A = (A_i,\phi_i)_{i \in Q_0}\in\Rep_{\B}(Q)$. Then $W = \bigoplus_{i \in Q_0} \im(g_i)$ naturally decomposes, where
\[\Hom(\bigoplus_{i \in Q_0} A_i, W) \isoto \prod_{i \in Q_0} \Hom(A_i,W),\ g \longmapsto (g_i)_{i \in Q_0},\]
and $g$ induces an isomorphism $A \isoto (\im(g_i), g_i \circ \phi_i \circ g_i^{-1})_{i \in Q_0}$ of quiver representations.
\item Similarly, consider the forgetful functor on the category of representations of $G$ in $\B$, for any group $G$. If $f\!: V \isoto W$ is an isomorphism in $\B$, then $\bee G \to \B,\ * \mapsto W$, with the action $g \mapsto f \circ \rho(g) \circ f^{-1}$, is isomorphic to $\rho\!: \bee G \to \B,\ * \mapsto V$, via the induced map by $f$.

By applying this example pointwise, we can extend it to representations of (affine) group schemes $G$ over $K$. Moreover, we deduce that the fibre functor of a neutral Tannakian category over $K$ is an (exact, $K$-linear) isofibration. This further generalizes to quasi-Tannakian categories (see \cite{OS}; for a brief summary, see \cite{An}, $\mathsection$2.1.4).

Indeed, if $X$ is an integral affine scheme over $K$ with $G$-action, together with a $K$-rational fixed point $x \in X(K)^G$, then the fibre functor $\rep_X(G) \to \vect_K$ at $x$ is an isofibration. Namely, one can show directly that every equivariant vector bundle on $X$ is of the form $V \otimes_K \O_X$ for some $V\in\rep_K(G)$. That is, it once again suffices to lift the action.

Note that the example $\omega\!: \Fil{[n]}{K} \to \vect_K$ from \ref{isofiba} above is a special case of this.
\end{enumerate}
\end{ex}

\begin{lemma}\label{K0(ratlSdot)}
Let $n\in\NN$. The \'etale stackification of the functor
\[\Aff_K \longto \Grpd,\ \Spec R \longmapsto K_0(S_n(\hat\E_R)),\]
defines a constant group scheme over $K$.
\end{lemma}
\begin{proof}
By Waldhausen's additivity theorem (\cite{Wa}, Lemma 1.4.3; also cf. $\mathsection$1.8), it suffices to show constancy of the functor
\[\K_0(\E)\!: \Spec R \longmapsto K_0(\hat\E_R).\]
Indeed, by the theorem of Murre \cite{Art}, Lemma 4.2, the resulting \'etale stackification $\ubar{K_0(\E)}$ will be a scheme. Now, there is a natural map $\ubar{K_0(\E)} \to \K_0(\E)$. Conversely, if $R$ is a local ring, the objects of $\E \hatotimes_K R$ can be written as $A \otimes_K R^{\oplus r}$, and the inverse map is given by
\[K_0(\E \hatotimes_K R) \longto K_0(\E),\ [A \otimes_K R^{\oplus r}] \longmapsto r \cdot [A].\]
\end{proof}

Let $\C$ be an $\infty$-category with finite limits and disjoint coproducts, and let $\V$ be a symmetric monoidal category with direct sums $\bigoplus$, which distribute over the tensor product.

\begin{defn}
A transfer theory $\T\!: \C \longto \V$ is $\sigma$-additive, if it preserves small coproducts,
\begin{equation}\label{coprodIso}
\T\big(\coprod_{i\in I} \Z_i\big) \isoto \bigoplus_{i\in I} \T(\Z_i),
\end{equation}
and for any collection $(f_i)_{i\in I}$ of $\T$-smooth, resp. $\T$-proper, morphisms, such that $\coprod f_i$ is again $\T$-smooth, resp. $\T$-proper, the identification \eqref{coprodIso} is functorial with respect to their pullback, resp. pushforward, and the diagram
\[\begin{tikzcd}
  \T(\coprod_{i\in I} \Z_i) \otimes \T(\coprod_{j\in J} \Z'_j) \ar[r, "{\boxtimes}"] \ar[d, "{\sim}"] & \T(\coprod_{i\in I} \Z_i \times \coprod_{j\in J} \Z'_j) \ar[r, "{=}"] & \T(\coprod_{i,j} \Z_i \times \Z'_j) \ar[d, "{\sim}"] \\
  \bigoplus_{i\in I} \T(\Z_i) \otimes \bigoplus_{j\in J} \T(\Z'_j) \ar[r, "{=}"] & \bigoplus_{i,j} \T(\Z_i) \otimes \T(\Z'_j) \ar[r, "{\boxtimes}"] & \bigoplus_{i,j} \T(\Z_i \times \Z'_j)
\end{tikzcd}\]
commutes, for all $\Z_i,\Z'_j \in \C$.
\end{defn}

\begin{ex}
The equivariant motivic transfer theory $K_0^{\D}(\Sta -)$ is $\sigma$-additive. The isomorphism \eqref{coprodIso} is explained in Remark \ref{eqTrId}, and the corresponding compatibilities follow from the fact that disjoint unions commute with fibre products.
\end{ex}

Now take $\C = \St K$, and let $\T\!: \C \dashto \V$ be a $\sigma$-additive transfer theory.

\begin{defn}
We say that $\Z\in\St K$ is measurable if its structure map $\Z \to \Spec(K)$ is both a $\T$-smooth and $\T$-proper morphism. If $\Z$ is measurable, we write
\[\mu_{\T}(\Z) := (\Z \to \Spec K)_* \circ (\Z \to \Spec K)^* \in \End(\T(\Spec K)) \isom \End(1_{\V}).\]
\end{defn}

\begin{rem}
Let $(\Z_i)_{i\in I}$ be a set of measurable stacks, such that $\coprod_{i\in I} \Spec(K)$ is again measurable. Then
\begin{equation}\label{muTadditive}
{\T}\big(\coprod_{i\in I} \Z_i\big) = \sum_{i\in I} \mu_{\T}(\Z_i) \circ \nabla^*,
\end{equation}
where $\nabla\!: \coprod_{i\in I} \Spec(K) \to \Spec(K)$ is the codiagonal. Indeed, \eqref{muTadditive} says that the diagram
\[\begin{tikzcd}
  \T(\Spec K) \ar[r, "{\nabla^*}"] \ar[dr] & \T(\coprod \Spec K) \ar[r, "{=}"] & \bigoplus \T(\Spec K) \ar[d, "{\bigoplus (\Z_i \to \Spec K)^*}"'] \\
  \bigoplus \T(\Z_i) \ar[d, "{\bigoplus (\Z_i \to \Spec K)_*}"] & \ar[l, "{=}"'] \T(\coprod \Z_i) \ar[dr] & \ar[l, "{=}"'] \bigoplus \T(\Z_i) \\
  \bigoplus \T(\Spec K) \ar[r, "{=}"] & \T(\coprod \Spec K) \ar[r, "{\nabla_*}"] & \T(\Spec K),
\end{tikzcd}\]
commutes, which in turn follows from $\sigma$-additivity of $\T$.
\end{rem}

\begin{ex}
If $\T = K_0(\Sta -)$, then $\Z\in\St K$ is measurable if and only if $\Z\in\Sta K$, and
\[\mu_{\T}(\Z) = [\Z] \in K_0(\Sta K).\]
For $\T = \mu_{\#}^*\QQ[\pi_0(-)]$, a stack $\Z$ is measurable if and only if $\Z(K)$ is finite, that is, $\pi_0(\Z(K))$ and all of its automorphism groups are finite. As expected, in this case,
\[\mu_{\T}(\Z) = \#\Z(K).\]
\end{ex}

\begin{defn}
Let $T$ be an algebra object in $\V$, with $\zeta\in\Aut(T)$. Further, let $\Gamma$ be a group, and let $\chi\!: \Gamma \times \Gamma \longto \ZZ$ be an integral $2$-cocycle, that is, satisfying the relation
\[\chi(\alpha,\beta) + \chi(\alpha\beta,\gamma) = \chi(\beta,\gamma) + \chi(\alpha,\beta\gamma)\]
for all $\alpha,\beta,\gamma\in\Gamma$. The corresponding twisted group ring $T^{\pair{\zeta,\chi}}[\Gamma]$ has underlying object
\[T[\Gamma] = \bigoplus_{\Gamma} T \in \V,\]
with natural inclusions $\iota_{\alpha}\!: T \to T[\Gamma]$ for $\alpha\in\Gamma$. The multiplication on $T^{\pair{\zeta,\chi}}[\Gamma]$ is given by
\[\sum_{\alpha,\beta\in\Gamma} \iota_{\alpha\beta} \circ \zeta^{\chi(\alpha,\beta)} \circ m \in \Hom\!\big(\!\bigoplus_{\Gamma\times\Gamma} T \otimes T, T[\Gamma]\big) = \Hom(T[\Gamma] \otimes T[\Gamma], T[\Gamma]),\]
where $m\!: T \otimes T \longto T$ denotes the multiplication on $T$. For us, $\Gamma$ will always be an abelian group, and $\chi$ bilinear. In this case, the unit for $T^{\pair{\zeta,\chi}}[\Gamma]$ is the composition $1_{\V} \to T \xto{\iota_1} T[\Gamma]$. In general, there is an additional factor of $\zeta^{-\chi(1,1)}\!: T \to T$.
\end{defn}

\begin{ex}
The Euler form of the $K$-linear quasi-abelian (or any exact) category $\E$ is the bilinear map
\[\chi\!: K_0(\E) \times K_0(\E) \longto \ZZ,\ ([A],[B]) \longmapsto \sum_{i\in\NN} (-1)^i\rk_K\Ext^i_{\E}(A,B).\]
\end{ex}

Suppose $\T\!: \St K \dashto \V$ is a $\sigma$-additive transfer theory such that $\S(\E)$ is $\T$-transferable. We abbreviate the corresponding Hall algebra by $\H(\E,\T) := \H(\S(\E),\T)$.

\begin{thm}\label{intthm}
Suppose $\M_{\E}$ is $\T$-proper. There exists a natural map of simplicial stacks
\begin{equation}\label{integration}
\oint^{(\bullet)}\!: \S(\E) \longto \ubar{K_0(S(\E))},
\end{equation}
the pushforward in degree $1$ along which defines a morphism
\begin{equation}\label{int}
\H(\E,\T) \longto \T(\Spec K)[K_0(\E)].
\end{equation}
Assume that $(\del_2,\del_0)$ is $\T$-proper, and that $[\AA^1_K/\Ga]$ is measurable with
\begin{equation}\label{probability}
\mu_{\T}([\AA^1_K/\Ga]) = 1.
\end{equation}
If $\E$ is hereditary, then \eqref{int} defines an algebra morphism
\[\int_{\E} -\ d\mu_{\T}\!: \H(\E,\T) \longto \T(\Spec k)^{\pair{\zeta,\chi^{\op}}}[K_0(\E)]\]
where $\zeta = \mu_{\T}(\AA^1_K)^{-1}$ denotes the Tate twist.
\end{thm}
\begin{proof}
The map $\oint^{(\bullet)}$ arises from the map of simplicial categories
\[S(\E) \longto \trl{L} S^{\pair 2}(\E),\]
which is induced by \cite{higherS}, $(4.8)$, as explained there. Passing to moduli stacks of objects, we obtain a map
\[\S(\E) \longto \trl{L} \S^{\pair 2}(\E).\]
Truncation to $\pi_0$ on the right yields the desired map of simplicial stacks
\[\oint^{(\bullet)}\!: \S(\E) \longto \ubar{K_0(S(\E))},\]
by Lemma \ref{K0(ratlSdot)}. Now, \cite{higherS}, Corollary 4.20, tells us that this discrete stack is in fact Segal. In particular, and more precisely, the diagram
\[\begin{tikzcd}
  \ubar{K_0(S_2(\E))} \ar[d, "{\sim}", "{(d_2,d_0)}"'] \ar[r, "{d_1}"] & \ubar{K_0(S_1(\E))} \\
  \ubar{K_0(S_1(\E))} \times_K \ubar{K_0(S_1(\E))} \ar[ur, dashed, "{\oplus}"'] & 
\end{tikzcd}\]
commutes. This means that applying $\T$ yields the (non-twisted) group ring $\T(\Spec K)[K_0(\E)]$. In order to understand the occurence of the twist, consider the square of correspondences
\[\begin{tikzcd}
  \S_1(\E) \times_K \S_1(\E) \ar[d, "{\oint\times_K\oint}"'] & \ar[l, "{(\del_2,\del_0)}"'] \S_2(\E) \ar[r, "{\del_1}"] & \S_1(\E) \ar[d, "{\oint}"] \\
  \ubar{K_0(S_1(\E))} \times_K \ubar{K_0(S_1(\E))} & \ar[l, "{(d_2,d_0)}", "{\sim}"'] \ubar{K_0(S_2(\E))} \ar[r, "{d_1}"'] & \ubar{K_0(S_1(\E))}
\end{tikzcd}\]
with the vertical identity maps omitted, where $\oint = \oint^{(1)}$ is the integral in degree $1$. In order to compare the two compositions, we take the fibre product on the left, as follows.
\[\begin{tikzcd}
  \S_1(\E) \times_K \S_1(\E) \ar[dd, "{\oint\times_K\oint}"'] & & \ar[ll, "{(\del_2,\del_0)}"'] \S_2(\E) \ar[r, "{\del_1}"] \ar[dl, dashed, "{\delta}"'] \ar[dd, "{\oint^{(2)}}"] & \S_1(\E) \ar[dd, "{\oint}"] \\
   & \ar[dl, phantom, "{\pullbackop}"' very near start] \ar[ul, "{\epsilon}", "{\sim}"'] \Z \ar[dr, "{\sigma}"'] & & \\
  \ubar{K_0(S_1(\E))} \times_K \ubar{K_0(S_1(\E))} & & \ar[ll, "{(d_2,d_0)}", "{\sim}"'] \ubar{K_0(S_2(\E))} \ar[r, "{d_1}"'] & \ubar{K_0(S_1(\E))}
\end{tikzcd}\]
Then we can compare the pullback-pushforward along the two edges of the (outer) square,
\begin{equation}\label{twist}
\oint_*(\del_1)_*(\del_2,\del_0)^* = (d_1)_*\oint^{(2)}_*\delta^*\epsilon^* = (d_1)_*\sigma_*\delta_*\delta^*\epsilon^*,
\end{equation}
by commutativity, and on the other hand, by compatibility with base change,
\begin{equation}\label{noTwist}
(d_1)_*(d_2,d_0)^*\Big(\oint\times_K\oint\Big)_* = (d_1)_*\sigma_*\epsilon^*.
\end{equation}
Of course, $\epsilon^*$ is an isomorphism, so the discrepancy is precisely $\delta_*\delta^*$. To analyze this, we use that $\oint$ is a graded map, meaning in particular that $\delta$ restricts to the fibres, say
\begin{equation}\label{deltaab}
\delta_{\alpha,\beta}\!: F_{\alpha,\beta}^{(2)} \longto F_{\alpha,\beta},
\end{equation}
of $\oint^{(2)}$ and $\oint\times_K\oint$ respectively, over $(\alpha,\beta) \in K_0(S_1(\E)) \times K_0(S_1(\E)) \isom K_0(S_2(\E))$. Further consider the fibres of $\delta_{\alpha,\beta}$ over $(A,B) \in F_{\alpha,\beta}(K) \subseteq \E^{\simeq} \times \E^{\simeq}$, which we denote by
\[\begin{tikzcd}
  E_{A,B} \ar[r, hook] \ar[d, "e"'] \ar[dr, phantom, "{\pullback}" very near start] & F_{\alpha,\beta}^{(2)} \ar[d, "{\delta_{\alpha,\beta}}"] \\
  \Spec(K) \ar[r, "{(A,B)}"'] & F_{\alpha,\beta}.
\end{tikzcd}\]
Then as a special case of \cite{Br}, Proposition 6.2, the $E_{A,B}$ are quotient stacks of the form
\[E_{A,B} \cong [\AA^{\rk\Ext^1(B,A)}_K/\Ga^{\rk\Hom(B,A)}],\]
as in our assumption \eqref{probability}. This result essentially goes back to \cite{Rie}. Since $\E$ is hereditary, this means that
\[e_*e^* = \zeta^{\chi^{\op}(\alpha,\beta)} \in\T(\Spec K),\]
for all $(A,B) \in F_{\alpha,\beta}(K)$. Now, for any scheme $X \xto{f} F_{\alpha,\beta}$, consider the stratification
\[X = \coprod_{d\in\NN} X_d\]
by the closed subschemes $X_d \subseteq X$, locally on $\Spec(R) \in \Aff_K$ defined by the equation
\[d = \rk_R \Hom((\pr_2 \circ f_R)(x),(\pr_1 \circ f_R)(x)), \mbox{ for } x \in X(R).\]
Then the fibres are constant above the (non-empty) strata, namely given by
\[(\delta_{\alpha,\beta} \times_{F_{\alpha,\beta}} X)^{-1}(X_d) \isom [\AA^{d-\chi^{\op}(\alpha,\beta)}_{X_d}/\GG_{a,X_d}^{d}] \isom E_{A,B} \times_K X_d,\]
for $(A,B) \in F_{\alpha,\beta}(K)$ with $\rk \Hom(B,A) = d$. Therefore, by $\sigma$-additivity of $\T$, we get
\[(\delta_{\alpha,\beta} \times_{F_{\alpha,\beta}} X)_*(\delta_{\alpha,\beta} \times_{F_{\alpha,\beta}} X)^* = (e \times_K \id_X)_*(e^* \btimes \id_X^*) = e_*e^* \otimes \id_{\T(X)} = \zeta^{\chi^{\op}(\alpha,\beta)}\cdot\id_{\T(X)}\]
for all $(A,B)$. Therefore indeed, for all $\alpha,\beta\in K_0(\E)$, the discrepancy is linear,
\[(\delta_{\alpha,\beta})_*(\delta_{\alpha,\beta})^* = \zeta^{\chi^{\op}(\alpha,\beta)}\cdot\id_{\T(F_{\alpha,\beta})} \in \End_{\T(\Spec K)}(\T(F_{\alpha,\beta})).\]
\end{proof}

\begin{rem}
As is evident from the definition (and used in the proof above), the integral in fact respects the $K_0(\E)$-gradings. Namely, by $\sigma$-additivity, the Hall algebra is graded by
\begin{equation}\label{K0grading}
\H(\E,\T) \isom \bigoplus_{\alpha\in K_0(\E)} \T(\M_{\alpha}), \mbox{ where } \M_{\alpha} := \oint^{-1}\alpha.
\end{equation}
Therefore, the integration map extends to the corresponding adic completions
\[\int_{\E} d\T\!: \hat{\H}(\E,\T) \longto \T(\Spec K)^{\pair{\zeta,\chi^{\op}}}[[K_0(\E)]]\]
along the augmentation ideals, generated by effective classes $\alpha\in K_0^+(\E)$. Note also that it surjects onto the twisted monoid algebra $\T(\Spec k)^{\pair{\zeta,\chi^{\op}}}[K_0^+(\E)]$ on the cone of effective classes in $K_0(\E)$. On the other hand, the integral does not distinguish between split and non-split extensions in $\E$ (up to scalar).

\end{rem}

\begin{ex}
For the transfer theory $\T = K_0^{\D}(\Sta -)$, we obtain the equivariant integral
\[\int_{\E}^{\D}\!: \H^{\D}(\E) \longto K_0^{\D}(\Sta K)^{\pair{\LL^{-1},\chi^{\op}}}[K_0(\E)].\]
In the case $\D = 0$, and thus $\T = K_0(\Sta -)$, this recovers the integration map
\[\int_{\E}\!: \H(\E) \longto K_0(\Sta K)^{\pair{\LL^{-1},\chi^{\op}}}[K_0(\E)],\ [\X] \longmapsto (\M_{\alpha} \to \Spec K)_*[\X]\cdot\alpha \mbox{ on } K_0(\Sta \M_{\alpha}),\]
from \cite{Joy2}, Theorem 6.4, or \cite{BehPo}, Definition 4.6 and Proposition 4.7. Let $\mu\!: K_0(\Sta K) \to Q$ be a motivic measure. Then we can extend scalars everywhere in $K_0(\Sta -)$ to produce a transfer theory $Q \otimes_{K_0(\Sta K)} K_0(\Sta -)$ valued in $Q$-modules. In this case, $\mu_{\T} = \mu$, and
\[\int_{\E} d\mu\!: \H(\E) \longto Q \otimes_{K_0(\Sta K)} \H(\E) \longto Q^{\pair{\mu(\LL^{-1}),\chi^{\op}}}[K_0(\E)],\ [\X\to\M_{\alpha}] \longmapsto \mu(\X)\cdot\alpha.\]
Similarly, for any of the equivariant motivic measures $\mu$ of Example \ref{eqmotmeas}, we obtain
\[\int_{\E}^{\D} d\mu\!: \H^{\D}(\E) \longto Q^{\pair{\mu(\LL^{-1}),\chi^{\op}}}[K_0(\E)].\]
If moreover $K=\FF_q$, and we apply $\T = \bbmu_{\#}^*\QQ[\pi_0(-)]$ instead, then by Remark \ref{motRHA}, we obtain
\[\sum\nolimits_{\E}\!: \Hall(\E) \longto \QQ^{\pair{q^{-1},\chi^{\op}}}[K_0(\E)],\ \mathbbm 1_A \longmapsto \sum_{B\in\oint(K)^{-1}[A]} \frac{\mathbbm 1_A(B)}{\#\Aut(B)}[B] = \frac{1}{\#\Aut(A)}[A].\]
This is Reineke's integration map, originating from \cite{Rei}, Lemma 6.1, and spelled out in \cite{Rei2}, Lemma 3.3. The compatibility
\[\begin{tikzcd}
  \H(\E) \ar[r, "{\int_{\E}}"] \ar[d, two heads] \mathclap{\smash{\mbox{\hspace{-3.8em}\raisebox{-3.7ex}{\scalebox{0.7}{\ensuremath{\bbmu_{\#}}}}}}} & K_0(\Sta K)^{\pair{\LL^{-1},\chi^{\op}}}[K_0(\E)] \ar[d, two heads, "{\mu_{\#}}"] \\
  \Hall(\E) \ar[r, "{\sum\nolimits_{\E}}"] & \QQ^{\pair{q^{-1},\chi^{\op}}}[K_0(\E)]
\end{tikzcd}\]
is a special case of Lemma \ref{countMeas}, which automatically implies that $\bbmu_{\#}^*\QQ[\pi_0(-)]$ satisfies the hypotheses of Theorem \ref{intthm} in the first place, which in turn are easily checked for $K_0^{\D}(\Sta -)$.

For $L|K$ finite, we can in fact extract from Lemma \ref{countMeas} (and Remark \ref{mutimesL}) a commutative square as follows,
\[\begin{tikzcd}
  \H(\E) \ar[r, "{\int_{\E}}"] \ar[d, two heads] \mathclap{\smash{\mbox{\hspace{-6em}\raisebox{-3.7ex}{\scalebox{0.7}{\ensuremath{\bbmu_{\#} \otimes_K L}}}}}} & K_0(\Sta K)^{\pair{\LL^{-1},\chi^{\op}}}[K_0(\E)] \ar[d, two heads, "{\mu_{\#} \otimes_K L}"] \\
  \Hall(\hat\E_L) \ar[r, "{\oint(L)_*}"] & \QQ^{\pair{q^{-[L:K]},\chi^{\op}}}[K_0(\E)].
\end{tikzcd}\]
\end{ex}

\begin{ex}
In the simplest case $\E = \vect_{\FF_q}$, Reineke's integration map is given by
\[\sum\nolimits_{\E}\!: \Hall(\E) \isoto \QQ^{\pair{q^{-1},\chi^{\op}}}[T],\ \mathbbm 1_{\FF_q^d} \longmapsto \frac{1}{(q^d-1)\cdots(q^d-q^{d-1})}T^d.\]
The right-hand side carries the product $T^c \cdot T^d = q^{-cd}T^{c+d}$. Of course, it is in fact further isomorphic to a non-twisted polynomial ring. Moreover, this generalizes to $\E=\Fil{[n]}{\FF_q}$ as
\[\sum\nolimits_{\E}\!: \Hall(\E) \isom \Hall(\vect_{\FF_q})^{\otimes n} \isoto \QQ^{\pair{q^{-1},\chi^{\op}}}[T_1,\ldots,T_n],\]
where the first isomorphism is given by $\mathbbm 1_{(V,F^{\bullet})} \mapsto \mathbbm 1_{F^1V} * \cdots * \mathbbm 1_{F^nV}$. In fact, in the limit,
\[\sum\nolimits_{\Fil{\ZZ}{\FF_q}}\!\!\!: \Hall(\Fil{\ZZ}{\FF_q}) \isom \colim \Hall(\vect_{\FF_q})^{\otimes n} \isoto \QQ^{\pair{q^{-1},\chi^{\op}}}[T_n \mid n\in\ZZ].\]
\end{ex}

\begin{ex}
Theorem \ref{intthm} also applies to the cohomological Hall algebra constructed by Kontsevich-Soibelman \cite{CoHA}. First, note that if $m > 0$, then the $\l$-adic cohomology groups
\[H^{\bullet}([\AA^n_K/\Ga^m],\Qlbar) = H^{\bullet}(\bee \Ga^m,\Qlbar) = \Sym(H^0(\Ga^m,\Qlbar)) \isom \Qlbar[t],\]
by base change and Borel's theorem (\cite{Beh}, Theorem 5.6)\footnote{This is Theorem 6.1.6 in the online version (available at \url{http://www.math.ubc.ca/~behrend/ladic.pdf}).}, whereas of course, when $m=0$, the cohomology $H^{\bullet}([\AA^n_K/\Ga^m],\Qlbar) = \Qlbar$ is trivial. However, for $m>0$, the map $\mu_{\T}([\AA^n_K/\Ga^m])$ is given by the inclusion $\Qlbar \into \Qlbar[t]$ in degree $0$, followed by the projection $t \mapsto 0$, and hence does not depend on $n,m$ whatsoever, as an element $\mu_{\T}([\AA^n_K/\Ga^m])\in H^{\bullet}(\Spec(K),\Qlbar) = \Qlbar$.

The original formulation in \cite{CoHA} uses $\ZZ$-coefficients, which is possible over $\CC$, by GAGA. As explained in \textit{op.cit.}, $\mathsection$2.4, the equivariant cohomology groups
\begin{equation}\label{eqCoh}
H^{\bullet}(\bee \GL_d, \Qlbar) \longinto H^{\bullet}(\bee \Gm^d, \Qlbar) \isom \Qlbar[x_1,\ldots,x_d], \mbox{ with } x_i \mbox{ in degree } 2,
\end{equation}
via pullback along the embedding $\Gm^d \into \GL_d$ as the diagonal torus, and the image of \eqref{eqCoh} is given by the symmetric polynomials. Then the cohomological Hall algebra carries a shuffle product, see \textit{op.cit.}, Theorem 2. More specifically, for $\E = \vect_K$, it is an exterior algebra (\textit{op.cit.}, $\mathsection$2.5). For degree reasons, as above, the integration map is simply the projection
\[\bigoplus_{d\geq 0} H^{\bullet}(\bee \GL_d, \Qlbar) \longto \Qlbar[T],\ \sum_{\nu\in\NN^d} a_{\nu}x_1^{\nu_1}\cdots x_d^{\nu_d} \longmapsto a_0\cdot T^d.\]
The above generalizes to the more sophisticated version of the cohomological Hall algebra, with coefficients in a Tannakian category $\V$, from \textit{op.cit.}, $\mathsection$3.3. Note that our assumption that $\V$ be symmetric monoidal is no hindrance; we merely impose it for convenience.
\end{ex}

\begin{rem}
We could equivalently twist the Hall algebra by
\[a \bullet b := \zeta^{-\chi^{\op}(\alpha,\beta)}(a*b), \mbox{ for } \deg(a) = \alpha,\ \deg(b) = \beta,\]
to obtain an algebra morphism to the non-twisted group ring. This is actually done in a different context, namely in the example of quiver representations, in order to remove the dependence of the Hall algebra on the orientation of the quiver. However, for our applications, the present statement of Theorem \ref{intthm} is more convenient.

Now suppose $\gamma\!: K_0(\E) \to \Gamma$ is a morphism of abelian groups which factors the Euler form,
\[\begin{tikzcd}
  K_0(\E) \times K_0(\E) \ar[r, "{\gamma\times\gamma}"] \ar[dr, "{\chi^{\op}}"'] & \Gamma \times \Gamma \ar[d, dashed, "{\exists\psi}"] \\
  & \ZZ.
\end{tikzcd}\]
Then we can compose the integration map with the induced algebra morphism,
\[\int_\E d\T[\gamma]\!: \H(\E,\T) \longto \T(\Spec K)^{\pair{\zeta,\chi^{\op}}}[K_0(\E)] \xto{\gamma_*} \T(\Spec K)^{\pair{\zeta,\psi}}[\Gamma].\]
\end{rem}

\begin{ex}
\begin{enumerate}[label=$(\arabic*)$]
\item If $\E = \vect_X$, where $X$ is a smooth projective curve of genus $g$ over $K$, we may take for $\Gamma$ the numerical Grothendieck group of $X$. Then the Chern character
\[\gamma\!: K_0(\E) \longonto \Gamma \cong \ZZ^2,\ [E] \longmapsto (\rk(E),\deg(E)),\]
factors the Euler form via the pairing
\[\psi\!: \ZZ^2 \times \ZZ^2 \longto \ZZ,\ (x,y) \longmapsto \det(x,y) + (1-g)x_1y_1,\]
as follows from the Riemann-Roch theorem. Consequently, we obtain an algebra morphism
\[\int_\E d\T[\gamma]\!: \H(\E,\T) \longto \T(\Spec K)^{\pair{\zeta,\psi}}[t_1^{\pm 1},t_2^{\pm 1}].\]
\item In the case of quiver representations $\E = \rep_K(Q)$, we may want to use the map
\[\gamma\!: K_0(\E) \longonto \ZZ^{Q_0},\ [A] \longmapsto \underline{\dim}(A),\]
to replace the Euler form by its usual explicit description. Namely,
\[\psi\!: \ZZ^{Q_0} \times \ZZ^{Q_0} \to \ZZ,\ (x,y) \mapsto \sum_{i\in Q_0} x_iy_i - \sum_{e\in Q_1} x_{s(e)}y_{t(e)}.\]
\item Similarly, by Example \ref{abEnvEX} \ref{abEnvEX3}, the category $\E = \Fil{\Lambda}{K}$ for a totally ordered set $\Lambda$ is again quasi-abelian and hereditary, because its abelian envelope $\Rep^f_{K}(\Lambda)$ is. Then Example \ref{rkdegEX} \ref{rkdegEX3} provides another instance of $\gamma$.
\item More exotic examples of hereditary categories include the ($\QQ$-linear) isogeny categories of $1$-motives \cite{1Mot}, resp. commutative algebraic groups \cite{Brion}, over a field.
\end{enumerate}
\end{ex}

\section{Slope filtrations and Harder-Narasimhan recursion}\label{sec10}

We now introduce the main examples and the appropriate context for the study of slope filtrations we are interested in, following \cite{An} (also cf. \cite{CC}). First, let us state the following useful characterization of quasi-abelian categories via torsion pairs.

\begin{prop}[\cite{BvdB}, Proposition B.3]\label{abEnv}
An additive category $\E$ with (co-)kernels is quasi-abelian if and only if there exists a fully faithful embedding $\E \into \A$ into an abelian category, and a full subcategory $\A_{\tors} \subseteq \A$, such that $\Hom(T,E) = 0$ for all $T\in\A_{\tors}$ and $E\in\E$, and such that for every $A\in\A$, there exists a short exact sequence in $\A$ of the form
\begin{equation}\label{torSeq}
0 \to T \to A \to E \to 0,
\end{equation}
with $T\in\A_{\tors}$ and $E\in\E$. In this case, \eqref{torSeq} is unique, and the induced functors $A \mapsto T$, resp. $A \mapsto A/T$, are right, resp. left, adjoint to the respective embeddings.
\end{prop}

\begin{defn}
A morphism $f\!: A \to B$ with $\ker(f) = \coker(f) = 0$ is called a pseudo-isomorphism. Equivalently, $f$ is both a monomorphism and an epimorphism in $\E$.
\end{defn}

\begin{rem}
There already is a notion of pseudo-isomorphism for the category of finitely generated modules over an integral domain $K$, namely that all localizations of the (co-)kernel at prime ideals of $K$ of height $\leq 1$ vanish.

Of course, in an abelian category, a pseudo-isomorphism in our sense is just the same as an isomorphism. On the other hand, if we restrict to the quasi-abelian subcategory of torsion-free $K$-modules, the above condition implies ours.
\end{rem}

\begin{ex}\label{abEnvEX}
\begin{enumerate}[label=$(\arabic*)$]
\item Let $K$ be a field, and $X$ a smooth projective curve over $K$. Set $\E = \vect_X$. Then
\[\O_X \longonto \O_X(1)\]
is a pseudo-isomorphism which is not an isomorphism (of course, $\Hom(\O_X(1),\O_X)=0$). Thus $\E$ is not abelian. In fact, its abelian envelope $\A$ as in Proposition \ref{abEnv} is given by the category of coherent sheaves on $X$, and $\A_{\tors}$ is the full subcategory of torsion sheaves. Note, however, that this does not generalize to higher dimensional schemes $X$.
\item Let $Q$ be a quiver, and let $\E$ be a quasi-abelian category with abelian envelope $\A$. We denote by $\Rep_{\E}(Q) = \Fun(Q,\E)$ the category of representations of $Q$ in $\E$. This is again a quasi-abelian category, with abelian envelope $\Rep_{\A}(Q)$.
\item\label{abEnvEX3} Let $\E = \Fil{\Lambda}{\B}$ be the category of $\Lambda$-filtered objects in an abelian category $\B$, where $\Lambda$ is a totally ordered set. For $\lambda\in\Lambda$, consider the filtration $F_{\lambda}^{\mu}V = \left\{\begin{array}{ll} V & \mbox{ if } \mu \leq \lambda \\ 0 & \mbox{ if } \mu > \lambda \end{array}\right.$ on $V\in\B$. If $\lambda < \lambda'$, then
\[(V,F^{\bullet}_{\lambda}) \xto{\id_V} (V,F^{\bullet}_{\lambda'})\]
is a pseudo-isomorphism, but its ``inverse'' does not respect filtrations, hence it is not an isomorphism. Assume $\Lambda = [n] = \{0,\ldots,n\}$, and let $\A = \Rep_{\B}(A_n)$ be the category of representations in $\B$ of the linearly oriented quiver $A_n$ on $n$ vertices.

Then $\E$ embeds into $\A$ as $(V,F^{\bullet}) \mapsto (F^nV \into \cdots \into F^1V)$, and the torsion subcategory is given by $\A_{\tors} = \{V_1 \to \cdots \to V_{n-1} \to 0\}$. Indeed, any morphism
\[\begin{tikzcd}
    V_1 \ar[r] \ar[d] & \cdots \ar[r] & V_{n-1} \ar[r] \ar[d] & 0 \ar[d] \\
    W_1 \ar[r, hook] & \cdots \ar[r, hook] & W_{n-1} \ar[r, hook] & W_n.
\end{tikzcd}\]
must be zero. On the other hand, let $A \in \Rep_{\B}(A_n)$. Then taking successive pullbacks yields
\[\begin{tikzcd}
    V_1 \ar[r] \ar[d, hook] \ar[dr, phantom, "{\pullback}" very near start] & \cdots \ar[r] & V_{n-1} \ar[r] \ar[d, hook] \ar[dr, phantom, "{\pullback}" very near start] & 0 \ar[d, hook] \\
    A_1 \ar[r, "{f_1}"] & \cdots \ar[r, "{f_{n-2}}"] & A_{n-1} \ar[r, "{f_{n-1}}"] & A_n
\end{tikzcd}\]
so that $V_i \isom \ker(f_{n-1} \circ\cdots\circ f_i)$ for all $i = (n-1),\ldots,1$. Hence indeed, the cokernel
\[\begin{tikzcd}
	\im(f_{n-1} \circ\cdots\circ f_1) \ar[r, hook] & \cdots \ar[r, hook] & \im(f_{n-1} \circ f_{n-2}) \ar[r, hook] & \im(f_{n-1})
\end{tikzcd}\]
lies in $\E$. By taking the colimit over $n\in\NN$, we obtain the analogous statement for
\[\Fil{\NN}{\B} \longinto \Rep^f_{\B}(A_{\NN}) := \Fun^f(A_{\NN},\B).\]
In fact, a similar argument works for any totally ordered set $\Lambda$, which can be filtered by its finite totally ordered subsets. For the latter, in turn, the torsion subcategories of the corresponding quiver representation categories are given by the objects sending all maximal elements to $0$.
\end{enumerate}
\end{ex}

\begin{rem}\label{pqab}
In what follows, $\E$ is either quasi-abelian or what by our conventions would have to be called proto-quasi-abelian, that is, proto-abelian in the sense of Andr\'e \cite{An}, Definition 1.2.3. Instead, we adopt the following terminology.
\end{rem}

\begin{defn}
We define a slope category, resp. additive ($K$-linear) slope category, to be a category as in Remark \ref{pqab}, resp. a ($K$-linear) quasi-abelian category, $\E$ equipped with a rank function. Here, a rank function on $\E$ is a map $\rk\!: \pi_0(\E^{\simeq}) \to \NN$, with $\rk(A) = 0 \iff A=0$, which extends to a morphism on the Grothendieck group
\[\rk\!: K_0(\E) \to \ZZ.\]
A degree function on $\E$ with values in a totally ordered abelian group $\Lambda$ is a group morphism
\[\deg\!: K_0(\E) \to \Lambda,\]
with $\deg(A) \leq \deg(B)$ for all pseudo-isomorphisms $A \to B$ in $\E$.
\end{defn}

From now on, we fix a $K$-linear slope category $\E$. Note that this is in particular a finiteness condition, as the rank bounds the length of flags in $\E$. In fact, if $\E$ is abelian, then it is both artinian and noetherian.

\begin{rem}\label{K0(abEnv)}
It follows from the construction that $K_0(\A) = K_0(\E)$ in Proposition \ref{abEnv}. Hence, the rank function formally extends to $\A$ (albeit not as a rank function, per se). One might expect that $\A_{\tors} = \{T\in\A \mid \rk(T) = 0\}$. This is the case in Example \ref{rkdegEX} \ref{rkdegEX1}--\ref{rkdegEX3}, but not in general (consider either of the trivial torsion pairs on $\A$).
\end{rem}

\begin{defn}
Let $\deg\!: K_0(\E) \to \Lambda$ be a degree function. The slope of $A\in\E$ is
\[\mu(A) = \frac{\deg(A)}{\rk(A)} \in \ZZ^{-1}\Lambda := \Lambda_{\QQ} \disj \{\infty\},\]
where $\Lambda_{\QQ} = \Lambda \otimes_{\ZZ} \QQ \isom (\ZZ\minus\{0\})^{-1}\Lambda$. The set $\ZZ^{-1}\Lambda$ carries the total order induced by
\[\frac{\lambda}{n} \leq \frac{\mu}{m} \iff m\lambda \leq n\mu \mbox{ for } m,n\in\NN, \mbox{ and } \Lambda_{\QQ} < \infty.\]
An object $A\in\E$ is called (semi-)stable, if for all non-trivial subobjects $0 \subsetneq B \subsetneq A$,
\[\mu(B)\ _{\operatorname{(}}\!\!\leq_{\operatorname{)}} \mu(A).\]
The full subcategories of semistable, resp. stable, objects $A\in\E$ of slope $\mu(A)\in\{\lambda,\infty\}$ are denoted by $\ss\E_{\lambda}$ and $\s\E_{\lambda}$, respectively.
\end{defn}

\begin{rem}\label{slopeprops}
Many of the usual properties hold in this generality. For example, the slope is convex on short exact sequences, and $\Hom(\ss\E_{\lambda},\ss\E_{\mu}) = 0$ for all $\lambda > \mu$ (\cite{An}, Lemma 1.3.8)\footnote{This is Lemma 3.2.3 in the arXiv version (\url{https://arxiv.org/pdf/0812.3921v2.pdf}).}.
\end{rem}

\begin{ex}\label{rkdegEX}
\begin{enumerate}[label=$(\arabic*)$]
\item\label{rkdegEX1} Let $X$ be a smooth projective (geometrically connected) curve over $K$, and $\E = \vect_X$ the category of vector bundles on $X$. For $E\in\E$, of course $\rk(E) = \dim_{k(x)}(E \otimes k(x))$ means the dimension of the fibre at a point $x\in X$.

The usual notion of degree of a divisor induces $\deg\!: \Pic(X) \to \ZZ$, which in turn extends to $K_0(\E)$ via $\deg(E) := \deg(\det E)$. Note that the pseudo-isomorphism $\O_X \onto \O_X(1)$ is consistent with the definition of a degree function.

\item Let $Q = (Q_0,Q_1,s,t)$ be a finite, connected quiver. Let $\E = \rep_K(Q) := \Rep_{\vect_K}(Q)$ be the category of finite dimensional representations of $Q$ over $K$, where $K$ may also be $\FF_1$. Then $\E$ is endowed with
\[\rk\!: \pi_0(\E^{\simeq}) \longto \NN,\ A \longmapsto \sum_{i\in Q_0} \dim_K(A_i).\]
The Grothendieck group of $\E$ can be identified with the dimension vectors
\[\ubar{\dim}\!: K_0(\E) \isoto \ZZ^{Q_0},\ A \longmapsto (\dim(A_i))_{i\in Q_0}.\]
Therefore, since $\E$ is abelian, the degree functions on $\E$ are precisely given by
\[\Hom(K_0(\E),\Lambda) \isoto \Lambda^{Q_0},\ \deg_{\theta} \mapsto \theta, \mbox{ with } \deg_{\theta}(A) = \sum_{i\in Q_0} \theta_i\dim(A_i).\]
\item\label{rkdegEX3} By Example \ref{abEnvEX}, we can try to descend $(b)$ to $\Fil{[n]}{k}$. It turns out that the cone of degree functions in $\Hom(K_0(\Fil{[n]}{k}),\Lambda)$ is given by $\{\deg_{\theta} \mid \theta \in \Lambda_+^{n-1} \times \Lambda\}$. The choice of $\theta$ is then essentially equivalent to an embedding into $\E = \Fil{\Lambda}{k}$, with its ``universal'' degree function
\[\deg_\bullet\!: K_0(\E) \to \Lambda,\ (V,F^\bullet) \mapsto \sum_{\lambda\in\Lambda} \lambda\dim(\gr^{\lambda}V) = \sum_{i=1}^r \lambda_i\dim(\gr^{\lambda_i}V),\]
where $\gr^{\lambda}V = F^{\lambda}V/F^{\lambda+1}V$ are the graded pieces and $\lambda_i$ the jumps of the filtration. Indeed, under the other natural identification $\ubar{\dim}(\gr^{\bullet})\!: K_0(\E) \isoto \Lambda^n$, we get
\begin{equation}\label{thetaIso}
\Lambda_+^{n-1} \times \Lambda \isoto \{\eta\in\Lambda^n \mid \eta_1 \geq\cdots\geq \eta_n\},\ \theta \longmapsto (\theta_1 + \cdots + \theta_n, \ldots, \theta_{n-1} + \theta_n, \theta_n).
\end{equation}
\end{enumerate}
\end{ex}

For many more examples, see the survey article \cite{An} by Andr\'e.

\begin{prop}
There exists a (unique, functorial) filtration \[F^{\bullet}\!: \Lambda_{\QQ}^{\op} \times \E \to \E,\ (\lambda,A) \mapsto F^{\lambda}A,\] such that for each $A\in\E$, the associated flag
\begin{equation}\label{HNflag}
0 \subseteq F^{\lambda_1}A \subseteq \cdots \subseteq F^{\lambda_n}A = A
\end{equation}
is uniquely determined by having successive quotients semistable of decreasing slopes
\[\mu(\gr^{\lambda_1}A) = \lambda_1 > \cdots > \lambda_n = \mu(\gr^{\lambda_n}A).\]
Conversely, $F^{\bullet}$ uniquely determines the degree function via
\[\deg(A) = \sum_{\lambda\in\Lambda_{\QQ}} \lambda\rk(\gr^{\lambda}A).\]
\end{prop}
\begin{proof}
This is proved by a standard argument, originally due to Harder-Narasimhan \cite{HN} in the context of vector bundles. The general case is \cite{An}, Theorem 1.4.7.
\end{proof}

\begin{ex}\label{HNtriv}
The HN-filtration on $\E = \Fil{\Lambda}{K}$ from Example \ref{rkdegEX} $(c)$ is rather tautological. Namely, $(V,F^{\bullet})\in\E$ is semistable if and only if $F^{\bullet}$ has precisely one jump. However, note that $\deg_\bullet$ extends to the ``$K$-rational'' category $\Fil{\Lambda}{L|K} = \Fil{\Lambda}{L} \times_{\vect_L} \vect_K$. Then for $L \neq K$, the slope filtration on $\Fil{\Lambda}{L|K}$ is not at all trivial.
\end{ex}

\begin{rem}\label{HNprop}
Note that there are some slight subtleties in case that $\Lambda$ is not discrete, e.g. in the definition of the jumps $\lambda_1 > \cdots > \lambda_n$ of $F^{\bullet}$ in the HN-flag \eqref{HNflag}. For an explanation of the details, see \cite{An}, Definition 1.4.1.

Labeling the filtration by the slopes (rather than to settle for the individual flags \eqref{HNflag}) is responsible for the functoriality of $F^{\bullet}$. This observation is due to Faltings.

Keeping in mind Remark \ref{K0(abEnv)}, we can extend $F^{\bullet}$ to the abelian envelope $\A$ of $\E$, assuming
\[\A_{\tors} = \{T\in\A \mid \rk(T) = 0\}.\]
Namely, if $A\in\A$ has torsion part $T\subseteq A$, we can pull back the filtration $F^{\lambda}A := \rho^{-1}F^{\lambda}(A/T)$ under the projection $\rho\!: A \onto A/T$. Then the HN-flag of $A$ is given by
\[0 \subseteq T = F^{\infty}A \subseteq F^{\lambda_1}A \subseteq \cdots \subseteq F^{\lambda_n}A = A.\]
This respects the decreasing slope condition, since $\gr^{\lambda}A = \gr^{\lambda}(A/T)$, and by assumption,
\[\mu(\gr^{\infty}A) = \infty > \mu(\gr^{\lambda_1}A) = \lambda_1 > \cdots > \lambda_n = \mu(\gr^{\lambda_n}A).\]
The only concession to make is that the filtration is no longer separated (and lives on $\ZZ^{-1}\Lambda$).

Usually, the Harder-Narasimhan filtration is related to the Jordan-H\"older filtration, as follows. If $\ss\E_{\lambda}$ is abelian (thus, artinian and noetherian), with simple objects precisely $\s\E_{\lambda}$, then the Jordan-H\"older filtration on $\ss\E_{\lambda}$ yields a refinement of the HN-filtration on $\E$.

This is the case in Example \ref{rkdegEX}, $(a)$-$(c)$. For general criteria, see \cite{An}, Corollary 1.4.10, as well as \textit{op.cit.}, Proposition 2.2.11.
\end{rem}

\begin{ex}\label{isocEX}
Let $L|K$ be a finite, totally ramified extension of complete discretely valued fields of characteristic $(0,p)$ with perfect residue field. The examples over $K=\QQ_p$ are precisely $L = K(\mu_{p^n}(\bar K))$, for $n\in\NN$. Let $\sigma$ lift the Frobenius of the residue field to $K$.

An isocrystal over $K$ is a pair $(V,\phi)$ with $V\in\vect_K$ and $\phi$ a $\sigma$-semilinear automorphism
\[\phi\!: V \otimes_{K,\sigma} K \isoto V.\]
Denote by $\isoc_K$ their category. The fibre functor $\omega\!: \isoc_K \to \vect_K$ is an isofibration, by Example \ref{isofibEX} $(c)$. Consider the category of filtered isocrystals over $L|K$, that is,
\[\Fil{\ZZ}{L|K} \times_{\vect_K} \isoc_K = \Fil{\ZZ}{L} \times_{\vect_L} \isoc_K.\]
This category carries two natural degree functions. On the one hand, coming from $\isoc_K$,
\[\deg_{\sigma}(V,F^{\bullet},\phi) := -\val_p(\det\phi).\]
On the other hand, the function $\deg_{\bullet}$ from Example \ref{rkdegEX} $(c)$. We are interested in their sum,
\[\deg = \deg_{\sigma} + \deg_{\bullet}\]
particularly because Colmez and Fontaine \cite{CF} show that there is an equivalence of categories
\[\ss{(\Fil{\ZZ}{L} \times_{\vect_L} \isoc_K)}_{0} \isoto \rep_K^{\cris}(G_L).\]
Here, $\rep_K^{\cris}(G_L)$ is the category of crystalline representations over $K$ of the absolute Galois group $G_L$ of $L$. They are particularly well-behaved objects of study in the local Langlands program.

Filtered isocrystals arise in $p$-adic Hodge theory as crystalline cohomology groups (which become $K$-vector spaces after inverting $p$) of smooth projective varieties over $L$, whose reductions modulo $p$ are smooth projective over the residue field. They carry a natural Frobenius action, and obtain a Hodge filtration (over $L$) by comparison with de Rham cohomology.
\end{ex}

We now discuss the generalization of the degree function in Example \ref{isocEX}. Fix a $K$-linear exact category $\B$. Let $\nu\!: \D \to \B$ be an exact functor with values in a slope category $\D$, and let $\omega\!: \E \to \B$ be an exact, $K$-linear isofibration, such that $\rk_{\D}$ and $\rk_{\E}$ factor through $K_0(\B)$. It will be convenient to note that under these circumstances, if $\D=0$, then so is $\B$.

\begin{lemma}\label{doubledeg}
Let $\deg_{\D}$ and $\deg_{\E}$ be degree functions with values in $\Lambda$ on $\D$ and $\E$, respectively. Then their sum is a well-defined degree function
\[\deg = \deg_{\D} + \deg_{\E}\!: K_0(\D \times_{\B} \E) \longto \Lambda.\]
\end{lemma}
\begin{proof}
By \cite{higherS}, Lemma 4.21, the natural map $K_0(\D \times_{\B} \E) \isoto K_0(\D) \times_{K_0(\B)} K_0(\E)$ is an isomorphism. Thus, applying $\Hom(-,\Lambda)$ to the projections onto $K_0(\D)$, resp. $K_0(\E)$, induces inclusions from the cone of degree functions on $\D$, resp. $\E$, into $\Hom(K_0(\D \times_{\B} \E),\Lambda)$. Since clearly the projections preserve pseudo-isomorphisms, the degree function property is retained.
\end{proof}

Now assume moreover that $\D$ is split proto-exact. The category $\D \times_{\hat\B_L} \hat\E_L$ is again quasi-abelian, for any field extension $L|K$. Indeed, let $\E\into\A$ be the abelian envelope (from Proposition \ref{abEnv}), then so is the fully faithful embedding
\[\D \times_{\hat\B_L} \hat\E_L \longinto \D \times_{\hat\B_L} \hat\A_L.\]
Furthermore, by our assumption and Lemma \ref{doubledeg}, the map $\deg$ makes $\D \times_{\hat\B_L} \hat\E_L$ into a slope category, where $\rk\!: K_0(\D \times_{\hat\B_L} \hat\E_L) \to \ZZ,\ (N,A) \mapsto \rk_{\D}(N)$, is a well-defined rank function.

The corresponding moduli stack $\M := \ubar{\D} \times_{\M_{\B}} \M_{\E}$ is stratified by Grothendieck classes
\[\M = \coprod_{\substack{[N] \in \pi_0(\D^{\simeq}) \\ \alpha \in K_0(\E)}} \M_{[N],\alpha}\]
similarly to \eqref{K0grading}, where we set $\M_{[N],\alpha} = \emptyset$ unless $[\nu(N)] = \omega(\alpha) \in K_0(\B)$. The HN-type, given by the classes of the graded pieces of the HN-flag \eqref{HNflag}, refines this decomposition,
\begin{equation}\label{HNstrat}
\M_{[N],\alpha} = \coprod_{\tau\models([N],\alpha)} \M_{[N],\alpha}^{\tau}.
\end{equation}
That is, $\tau$ runs over all $(\tau_1,\ldots,\tau_n)\in K_0(\D \times_{\B} \E)^{\times n}$ with $\tau_1 + \cdots + \tau_n = ([N],\alpha)$ and such that $\tau_1 > \cdots > \tau_n$, which we write to abbreviate $\mu(\tau_1) > \cdots > \mu(\tau_n)$. Amongst these are in particular the $1$-step filtrations, where the functor $\M_{[N],\alpha}^\tau =: \ss\M_{[N],\alpha}$ parametrizes semistable objects of class $\tau = ([N],\alpha)\in K_0(\D) \times_{K_0(\B)} K_0(\E)$.

As explained on \cite{DOR}, p.196, the HN-strata $\M_{[N],\alpha}^{\tau}$ are not substacks of $\M_{[N],\alpha}$ in general. In addition to the usual semi-continuity of the slope, this would require elements of $\D$ to have finitely many subobjects. However, in the non-archimedean setting, this condition can be replaced by compactness of the moduli space of subobjects (cf. \textit{op.cit.}, Proposition 8.2.1, and Proposition 8.3.4\textit{f}). For this reason we consider the analytic variant of the Hall algebra. Let us assume from now on that $(\ubar{S_2(\D)}_{[M],[N]})^{\an}$ is compact for all $M,N \in \D$.

\begin{ex}\label{periodex}
\begin{enumerate}[label=$(\arabic*)$]
\item\label{periodFq} Consider the fibre functor $\omega\!: \Fil{[n]}{K} \to \vect_K$, as well as $\nu = \id_{\vect_K}$. Let us write $\M_{\alpha} = \M_{[K^r],\alpha}$ for $\alpha \in K_0(\Fil{[n]}{K})$ and $r=\rk(\alpha)$. Recall that the flag variety $\pi_0\M_{\alpha}$ is represented by the quotient scheme $\GL_r\!/P_{\alpha}$; in fact,
\begin{equation}\label{flagvar}
[\M_{\alpha}] = [\GL_r\!/P_{\alpha}] [\bee \ubar{P_{\alpha}(K)}] \in K_0^{\vect_K}(\Sta K),
\end{equation}
where $P_{\alpha} \subseteq \GL_r$ is the stabilizer of the standard flag of type $\alpha = (r_i)_{i\in\NN} \in K_0^+(\Fil{[n]}{K})$, with $r = r_n$, which is a parabolic subgroup. Namely,
\[\begin{tikzcd}
  \GL_r\!/P_{\alpha} = \pi_0\M_{\alpha} \ar[r, hook] & \prod\limits_{i=1}^n \Gr_{r_i,r} \ar[r, hook, "{\det}"] & \prod\limits_{i=1}^n \PP^{d_i}_K \ar[r, hook] & \PP^N_K,
\end{tikzcd}\]
where $\Gr_{r_i,r} = \GL_r\!/P_{r_i,r}$ is the corresponding Grassmannian, $\det$ the appropriate product of Pl\"ucker embeddings, with
\[d_i = \binom{r}{r_i} = \dim_K\!\big(\!\bigwedge\nolimits^{\!\!r_i}\!K^r\big),\]
and $N = \prod\limits_{i=1}^n (d_i+1)-1$. The $\vect_K$-action on $\GL_r/P_{\alpha}$ becomes left multiplication by $\GL_r(K)$. For $K = \FF_q$, the $\pi_0\ss\M_{\alpha}$ are open subvarieties of $\GL_r/P_{\alpha}$, and studied in \cite{DOR} as the analogues over the finite field $K$ of ($p$-adic) period domains.
\item\label{periodF1} Consider the same situation as in \ref{periodFq}, but with $\D = \vect_{\FF_1}$. Note that $\Hom_{\FF_1}(-,\FF_1)$ is an exact duality on $\D$ and every object is naturally identified with its dual. Thus, we can consider $\Hom_{\FF_1}(-,(K,0))$ as a covariant, exact functor
\[\nu\!: \vect_{\FF_1} \longto \vect_K,\]
for any field $K$. If $K = \FF_q$, we recover the "period domains over $\FF_1$" from \cite{DOR}.
\item\label{periodQp} The main example is provided by Example \ref{isocEX}, where the base field was a $p$-adic field $K$, and we considered the pair of fibre functors
\begin{equation}\label{isofibs}
\begin{tikzcd}
  \isoc_K \ar[r, "{\nu}"] & \vect_K & \ar[l, "{\omega}"'] \Fil{\ZZ}{K}.
\end{tikzcd}
\end{equation}
Similarly to \ref{periodFq}, we have an equivalence of functors, hence analytic spaces,
\[(\pi_0\ss\M_{[N],\alpha})^{\an} \isom \ss{\F(N,\gamma(\alpha))}\]
with the $p$-adic period domain of type $(N,\gamma(\alpha)$ from \textit{op.cit.}, Proposition 8.2.1. Here, the map $\gamma\!: K_0(\Fil{[n]}{K}) \isoto K_0(\vect_K)^{\oplus n}$ sends a filtration to its graded pieces.
\end{enumerate}
\end{ex}

%

\begin{prop}
In the equivariant motivic Hall algebra, \eqref{HNstrat} translates to the identity
\begin{equation}\label{HNrec}
[\M_{[N],\alpha}] = \sum_{\tau \models ([N],\alpha)} [\ss\M_{[N_1],\alpha_1}] * \cdots * [\ss\M_{[N_n],\alpha_n}],
\end{equation}
where $\tau_i = ([N_i],\alpha_i)$ for $1 \leq i \leq n$, and the right-hand side carries its natural $\D$-action
\[[\ss\M_{[N_i],\alpha_i}] \in K_0^{\Aut(N_i)}(\Sta \M_{\E}) \subseteq \H^{\D}(\E).\]
\end{prop}
\begin{proof}
First, assume $\D = 0$. Then \eqref{HNrec} follows immediately from \eqref{HNstrat} and by definition,
\begin{equation}\label{HNrecOBV}
[\ss\M_{\alpha_1}] * \cdots * [\ss\M_{\alpha_n}] = [(\ss\M_{\alpha_1} \times_K \cdots \times_K \ss\M_{\alpha_n}) \times_{(\M \times_K \cdots \times_K \M)} \S_n(\E)] = [\M_{\alpha}^{\tau}]
\end{equation}
using Remark \ref{slopeprops}. Next, suppose that $\D$ has finite automorphism groups. Then the main observation is that each stratum $\M_{[N],\alpha}^{\tau}$ of type $\tau$ is further stratified by the substacks $\M_{[N],\alpha}^{\tau,M_{\bullet}}$ parametrizing objects where $M_{\bullet}$ is the $\D$-component of the HN-flag. That is,
\[\M_{[N],\alpha}^{\tau} = \coprod_{[M_{\bullet}] = \eta} \M_{[N],\alpha}^{\tau,M_{\bullet}} = \!\coprod_{\bar{g}\in\Aut(N)/P_{\eta}(K)}\! \M_{[N],\alpha}^{\tau,\bar{g}N_{\bullet}}\]
where $\eta = \phi(\pr_{\D}(\tau))$, where $\phi$ is the inverse of $\gamma\!: K_0(\Fil{[n]}{\D}) \isoto K_0(\D)^{\oplus n}$, while $P_{\eta}(K)$ is the corresponding parabolic subgroup, and $N_{\bullet}$ is the filtration $\phi(N_1,\ldots,N_n)$ of $N$.

The strata are pairwise isomorphic, and the action is as expected, where in particular, each $h\in\Aut(N)$ interchanges $\M_{[N],\alpha}^{\tau,\bar g_1 N_{\bullet}}$ and $\M_{[N],\alpha}^{\tau,\bar g_2 N_{\bullet}}$ whenever $hg_1 \in g_2P_{\eta}(K)$. That is, the HN-stratum is given by parabolic induction (cf. \eqref{parabInd}) as follows,
\[\M_{[N],\alpha}^{\tau} = \Ind_{P_{\eta}(K)}^{\Aut(N)}(\M_{[N],\alpha}^{\tau,N_{\bullet}}) = \Ind_{P_{\eta}(K)}^{\Aut(N)}\big((\ss\M_{[N_1],\alpha_1} \times_K \cdots \times_K \ss\M_{[N_n],\alpha_n}) \times_{\M_{\E}^{\times n}} \S_n(\E)\big),\]
where the second equation can be seen in a similar manner as \eqref{HNrecOBV}, using that $\D$ is split proto-exact, and $\bar{g}\in\Aut(N)/P_{\eta}(K)$ embeds $N_{\bullet}$ into $N \isom N_1 \oplus \cdots \oplus N_n$ on the right-hand side. Commuting the parabolic induction with the fibre product, this allows us to deduce
\begin{equation}\label{HNrecstrat}
[\M_{[N],\alpha}^{\tau}] = [\ss\M_{[N_1],\alpha_1}] * \cdots * [\ss\M_{\alpha_n,[N_n]}]
\end{equation}
as claimed. Finally, in the general case, the argument is similar, except that we do not have the convenient description of the parabolic induction functor as above. Nonetheless, the map
\[[\M_{[N],\alpha}^{\tau,N_{\bullet}}/\ubar{P_{\eta}(K)}] \longto [\M_{[N],\alpha}^{\tau}/\ubar{\Aut(N)}] \mbox{ in } \St (\bee\ubar{\Aut(N)} \times_K \M_{\E})\]
is a geometric equivalence, and therefore, the same is true for the morphism
\[[\big((\ss\M_{[N_1],\alpha_1} \times_K \cdots \times_K \ss\M_{[N_n],\alpha_n}) \times_{\M_{\E}^{\times n}} \S_n(\E)\big)/\ubar{P_{\eta}(K)}] \longto [\M_{[N],\alpha}^{\tau}/\ubar{\Aut(N)}],\]
by the equivalence explained in the previous case. But we can identify the left-hand side with
\[\Big[\big(([\ss\M_{[N_1],\alpha_1}/\ubar{\Aut(N_1)}] \times_K \cdots \times_K [\ss\M_{[N_n],\alpha_n}/\ubar{\Aut(N_n)}]) \times_{\ubar{\D}^{\times n}} \ubar{S_n(\D)}\big) \times_{\M_{\E}^{\times n}} \S_n(\E)\Big].\]
As before, we may commute the fibre products, and conclude \eqref{HNrecstrat} once again.
\end{proof}

\begin{rem}
In the original situation \cite{Rei} over $\FF_q$, the analogue of \eqref{HNrec} is the equation
\[\mathbbm 1_{\E} = \sum_{\lambda_1 > \cdots > \lambda_n} \mathbbm 1_{\ss\E_{\lambda_1}} * \cdots * \mathbbm 1_{\ss\E_{\lambda_n}} \in \hat{\Hall}(\E).\]
Indeed, in the completed Hall algebra $\hat{\H}^{\D}(\E)$, we may sum over all $(\alpha,[N]) \in K_0(\E \times_{\B} \D)$.
\end{rem}

\begin{cor}
If the right-hand side converges, the following equation holds in $\hat{\H}^{\D}(\E)$,
\begin{equation}\label{HNinv}
[\ss\M_{[N],\alpha}] = \sum_{\tau \dom ([N],\alpha)} (-1)^{m-1} [\M_{[N_1],\alpha_1}] * \cdots * [\M_{[N_m],\alpha_m}],
\end{equation}
where $\tau_i = ([N_i],\alpha_i)$ for all $1 \leq i \leq m$, whereas $\tau \dom ([N],\alpha)$ means $\tau_1 + \cdots + \tau_m = ([N],\alpha)$ and $\tau_1 + \cdots + \tau_i > ([N],\alpha)$ for $i < m$.
\end{cor}
\begin{proof}
This inversion is due to Reineke \cite{Rei}, Theorem 5.1; also see \cite{Joy4}, Theorem 5.12, for a generalized version of it. The idea is simply to substitute \eqref{HNrec} into \eqref{HNinv}, to get
\begin{equation}\label{naivesub}
\sum_{(\tau_1,\ldots,\tau_m) \dom ([N],\alpha)}\!\! (-1)^{m-1} \!\!\!\!\sum_{(\tau_{i,1},\ldots,\tau_{i,n_i}) \models \tau_i}\!\! [\ss\M_{\tau_{1,1}}] * \cdots * [\ss\M_{\tau_{1,n_1}}] * \cdots * [\ss\M_{\tau_{m,1}}] * \cdots * [\ss\M_{\tau_{m,n_m}}].
\end{equation}
Note that by convexity of the slope $\mu$, we can equivalently describe
\[(\tau_1,\ldots,\tau_m) \dom ([N],\alpha) \longiff \tau_1 + \cdots + \tau_m = ([N],\alpha) \mbox{ and } \tau_1 + \cdots + \tau_i > \tau_{i+1} + \cdots + \tau_m \mbox{ for } i < m.\]
Using this characterization, it is evident that \eqref{naivesub} is of the form
\[\sum_{\tau \dom ([N],\alpha)} (-1)^{n_{\tau}-1} [\ss\M_{[N_1],\alpha_1}] * \cdots * [\ss\M_{[N_m],\alpha_m}],\]
for appropriate $n_{\tau}\in\NN$. Then it becomes a purely combinatorial statement that the only term which remains is indeed $[\ss\M_{[N],\alpha}]$, namely \cite{Rei}, Lemma 5.4.
\end{proof}

\begin{rem}
In order to circumvent the above convergence issues, we need to apply more sophisticated inversion formulae, like those due to Zagier \cite{Zag} and Laumon-Rapoport \cite{LR}, the latter based on the idea of Kottwitz to apply the Langlands lemma from the theory of Eisenstein series. In particular, \eqref{HNinv} applies in the case where $\omega\!: \E = \rep_K(Q) \to \D = 0$, as well as to Example \ref{periodex} \ref{periodFq}, but for instance not to $\omega\!: \vect_X \to \D = 0$.
\end{rem}

\begin{cor}\label{mu(HNinv)}
For any equivariant motivic measure $\mu\!: K_0^{\D}(\Sta K) \to Q$,
\[\mu(\ss\M_{[N],\alpha}) = \sum_{\tau \dom ([N],\alpha)} (-1)^{m-1} \mu(\LL)^{-\sum_{i < j} \chi^{\op}(\alpha_i,\alpha_j)} \mu(\M_{[N_1],\alpha_1}) \cdots \mu(\M_{[N_m],\alpha_n}) \in Q,\]
whenever the right-hand side is defined.
\end{cor}
\begin{proof}
This is the image of \eqref{HNinv} under $\int_{\E}^{\D} d\mu$, using that it is an algebra morphism.
\end{proof}

In particular, for $\D = 0$, we can recursively compute any motivic measure of $\ss\M_{\alpha}$, see Example \ref{motMeasEX}. Now, if $X\in\Var K$ carries the action of a group $H$, then by Example \ref{eqmotmeas}, its \'etale cohomology (with compact support) is a representation of $H$, so its virtual Euler-Poincar\'e characteristic is an element
\[\chi_c(X) \in K_0(\rep_{\Qlbar}(G_K \times H)).\]
Applying $\Fun(\bee H,-)$ in \eqref{K0VarK0St} allows us to extend this to the equivariant motivic measure
\[\chi_c\!: K_0^H(\Sta K) \longto K_0(\rep_{\Qlbar}\!(H \times G_K))[(t^{-n}-1)^{-1} \mid n\in\NN].\]
Altogether, for $\D$ as before, we can apply Corollary \ref{mu(HNinv)} to the morphism of rings
\begin{equation}\label{EPchar}
\chi_c\!: K_0^{\D}(\Sta K) \longto \bigoplus_{N\in\pi_0(\D^{\simeq})} K_0(\rep_{\Qlbar}\!(\Aut(N) \times G_K))[(t^{-n}-1)^{-1} \mid n\in\NN],
\end{equation}
where the right-hand side carries the product induced by parabolic induction. Indeed,
\[\rep_{\Qlbar}\!(H \times G) = \Fun(\bee H \times \bee G,\vect_{\Qlbar}) = \Fun(\bee H,\Fun(\bee G,\vect_{\Qlbar})) = \Fun(\bee H,\rep_{\Qlbar}\!(G))\]
implies that the right-hand side of \eqref{EPchar} is defined precisely like the left-hand side,
\[\bigoplus_{N\in\pi_0(\D^{\simeq})} K_0(\rep_{\Qlbar}\!(\Aut(N) \times G_K)) \isom K_0(\Fun^f(\D^{\simeq},\rep_{\Qlbar}\!(G_K))).\]
Thus, $\chi_c$ is compatible with multiplication, by Remark \ref{eqTrId}. For $\eta = \phi(\pr_{\D}(\tau))$, this yields
\[\chi_c(\ss\M_{[N],\alpha}) = \sum_{\tau \dom ([N],\alpha)} (-1)^{m-1} t^{\sum_{i < j} \chi^{\op}(\alpha_i,\alpha_j)} \Ind_{P_{\eta}(K)}^{\Aut(N)}(\chi_c(\M_{[N_1],\alpha_1}) \cdots \chi_c(\M_{[N_m],\alpha_m})).\]
Of course, the analogous statement holds for any equivariant motivic measure (Example \ref{eqmotmeas}). Now consider the situation of Example \ref{periodex} \ref{periodFq}. Then by \eqref{flagvar}, we obtain from the above
\[\chi_c(\pi_0\ss{\M_{\alpha}}) = \sum_{\tau \dom \alpha} (-1)^{m-1} t^{\sum_{i < j} \chi^{\op}(\alpha_i,\alpha_j)} \Ind_{P_{\eta}(K)}^{\GL_r(K)}(\chi_c(\GL_{\eta_1}\!/P_{\alpha_1}) \cdots \chi_c(\GL_{\eta_m}\!/P_{\alpha_m})).\]
But the Euler-Poincar\'e characteristic of the flag variety of type $\alpha$ is given by
\begin{equation}\label{EPchar(F)}
\chi_c(\GL_r/P_{\alpha}) = \frac{(t^r-1) \cdots (t-1)}{(t^{\delta_1}-1) \cdots (t-1) \cdots (t^{\delta_n}-1) \cdots (t-1)} = \sum_{w \in S_r/S_{\delta}} t^{\l(w)}
\end{equation}
in $K_0(\rep_{\Qlbar}\!(\GL_r(K) \times G_K))$, where $\delta = \gamma(\alpha)$, and the notation $S_{\delta}$ in \eqref{EPchar(F)} means the stabilizer of $\delta$ in the symmetric group. Note that unlike for $K_0(\Var K)$, the localization map here is injective.

Over $\FF_1$ (Example \ref{periodex} \ref{periodF1}), we get a completely analogous formula in $K_0(\rep_{\Qlbar}\!(S_r \times G_K))$,
\[\chi_c(\pi_0\ss{\M_{\alpha}}) = \!\sum_{\tau \dom \alpha}\! (-1)^{m-1} t^{\sum_{i < j} \chi^{\op}(\alpha_i,\alpha_j)} \Ind_{S_{\eta}}^{S_r}(\chi_c(\GL_{\eta_1}\!/P_{\alpha_1}) \cdots \chi_c(\GL_{\eta_m}\!/P_{\alpha_m})).\]

\section{Equivariant motivic Hall modules}\label{sec11}

The notion of Hall modules originates in \cite{Matt} and has been systematically developed in both \cite{Tashi} and \cite{Matt2}, based on the formalism of relative $2$-Segal objects. The framework of transfer theories from $\mathsection$\ref{sec7} applies verbatim, yielding the following result.

\begin{thm}[\cite{Tashi}, Corollary 1.3.9; \cite{Matt2}, Theorem 4.2]\label{Hallmod}
Let $\T\!: \C \dashto \V$ be a transfer theory, let $X$ be a $\T$-transferable unital $2$-Segal object of $\C$, and let $\lambda\!: Y \to X$ be a (unital) relative $2$-Segal object such that the correspondence
\[\begin{tikzcd}
X_1 \times Y_0 & \ar[l, "{(\lambda_1,\del_0)}"'] Y_1 \ar[r, "{\del_1}"] & Y_0
\end{tikzcd}\]
is $\T$-(smooth, proper). Then $\T(Y_0)$ is a left $\T(X_1)$-module via
\[\begin{tikzcd}
\T(X_1) \otimes \T(Y_0) \ar[r, "{\boxtimes}"] & \T(X_1 \times Y_0) \ar[r, "{(\lambda_1,\del_0)^*}"] & \T(Y_1) \ar[r, "{(\del_1)_*}"] & \T(Y_0).
\end{tikzcd}\]
\end{thm}

Throughout this section, $\E$ is an exact, $K$-linear category endowed with an exact duality structure (the setting of real algebraic $K$-theory; cf. \cite{HM}, \cite{Dotto}), which satisfies the reduction assumption on \cite{Matt2}, p.26.

Let $G = \Gal(\CC|\RR) = \{1,\sigma\}$. Consider the $G$-equivariant inclusion of the $k$-skeleton
\begin{equation}\label{integral}
S^{k,1} \wedge |\E^{\simeq}| \longinto |S^{\pair k}(\E)^{\simeq}|,
\end{equation}
where $S^{\pair k}(\E)$ is the higher Waldhausen construction \cite{higherS}, and $G$ acts as follows. For $k \geq e$, the sphere
\[S^{k,e} = (S^{1,0})^{\wedge(k-e)} \wedge (S^{1,1})^{\wedge e}\]
is the smash product of the $1$-point compactifications of $k-e$ copies of the trivial and $e$ copies of the sign representation of $G$ on $\RR$. The indexing is inspired by the spheres in the motivic homotopy category. Topologically, $S^{k,e}$ is a $k$-sphere, while its real points $(S^{k,e})^{\sigma}$ form a sphere of dimension $k-e$.

The duality structure on $\E$ also induces a $G$-action on the edge-wise subdivision $ES(\E)^{\simeq}$, which respects the simplicial structure. Then the canonical map
\[\lambda\!: (ES(\E)^{\simeq})^{\sigma} \longinto ES(\E)^{\simeq} \longto S(\E)^{\simeq}\]
defines a module structure over the Hall algebra on its $0$-cells (the self-dual objects in $\E$) via the correspondence
\begin{equation}
  \begin{tikzcd}
    S_1(\E)^{\simeq} \times (ES_0(\E)^{\simeq})^{\sigma} & \ar[l, "{(\lambda_1,\del_0)}"'] (ES_1(\E)^{\simeq})^{\sigma} \ar[r, "{\del_1}"] & (ES_0(\E)^{\simeq})^{\sigma}.
  \end{tikzcd}
\end{equation}
This recovers in particular the Ringel-Hall modules first discovered in \cite{Matt}.

\begin{thm}[\cite{Matt2}, Theorem 3.6]
The map $\lambda\!: (ES(\E)^{\simeq})^{\sigma} \longto S(\E)^{\simeq}$ is relative $2$-Segal.
\end{thm}

Passing to moduli stacks, we obtain the relevant property for our purposes.

\begin{cor}\label{rel2segal}
The map $\lambda\!: (E\S(\E))^{\sigma} \longto \S(\E)$ is relative $2$-Segal.
\end{cor}

\begin{defn}
Let $\D$ be a split proto-exact category. We define the equivariant motivic Hall module $\h^{\D}(\E)$ as the left $\H^{\D}(\E)$-module constructed in Theorem \ref{Hallmod} along the transfer theory $K_0^{\D}(\Sta -)$ via Corollary \ref{rel2segal}.

Similarly, if $K$ is non-archimedean, the analytic equivariant motivic Hall module $\h^{\D,\an}(\E)$ is the left $\H^{\D,\an}(\E)$-module defined analogously along the transfer theory $K_0^{\D}(\AnSta (-)^{\an})$.
\end{defn}

In order to construct the analogue of the integration map \eqref{integration} in this context, we consider the inclusions of $2$-skeleta, for all $n\in\NN$,
\[S^{2,1} \wedge |ES_n(\E)^{\simeq}| \longinto |S^{\pair 2}(ES_n(\E))^{\simeq}|.\]
Let $KR(\E) = \Omega^{2,1}|S^{\pair 2}(\E)|$ denote the real $K$-theory space of an exact category $\E$ with exact duality structure. Then
\begin{equation}\label{moduleint}
ES(\E)^{\simeq} \longto \pi_0 KR(ES(\E))
\end{equation}
arises as in \eqref{integration} above. Taking fixed points, we obtain the map of simplicial groupoids
\[(ES(\E)^{\simeq})^{\sigma} \longto \pi_0 KR(ES(\E))^{\sigma} = GW_0(ES(\E)).\]
The Grothendieck-Witt space $GW(\E)$ is shown to be weakly equivalent to $KR(\E)^{\sigma}$ in \cite{HM}.

Again, we may pass to moduli stacks to obtain the module integration map
\[\oint^{\sigma,(\bullet)}\!: (E\S(\E))^{\sigma} \longto \ubar{GW_0(ES(\E))}.\]
Now let $\T\!: \St K \dashto \V$ be a transfer theory such that the map $\lambda$ from Corollary \ref{rel2segal} satisfies the assumptions of Theorem \ref{Hallmod} with respect to $\T$. Assume that both $\M_{\E}$ and its substack of self-dual objects $\M_{\E}^{\sigma}$ are $\T$-proper. Then the pushforward along $\oint^{\sigma} = \oint^{\sigma,(1)}$ is a map
\begin{equation}\label{nontwistedmodint}
\oint^{\sigma}_*\!: \h(\E,\T) := \h(\S(\E),\T) \longto \T(\Spec K)[GW_0(\E)].
\end{equation}
If $\E$ is hereditary, $(\lambda_1,\del_0)\!: (E\S_1(\E))^{\sigma} \to \S_1(\E) \times (E\S_0(\E))^{\sigma}$ is $\T$-proper, and $[\AA^1_K/\GG_a]$ is measurable with $\mu_{\T}([\AA^1_K/\GG_a]) = 1$, then \eqref{nontwistedmodint} induces a $\H(\E,\T)$-module morphism
\[\int^{\sigma}_{\E}\!: \h(\E,\T) \longto \T(\Spec K)^{\pair{\zeta,\chi^{\op}+\chi^{\sigma}}}[GW_0(\E)].\]
Here, $\zeta$ is the Tate twist as in Theorem \ref{intthm}, and the self-dual Euler form is given by
\[\chi^{\sigma}\!: K_0(\E) \longto \ZZ,\ [A] \longmapsto \sum_{i\in\NN} (-1)^i\dim_K\Ext^i_{\E}(A^{\vee},A)^{(-1)^{i+1}\sigma},\]
see \cite{Matt}, Theorem 2.6. Then $\T(\Spec K)^{\pair{\zeta,\chi^{\op}+\chi^{\sigma}}}[GW_0(\E)]$ is a $\T(\Spec k)^{\pair{\zeta,\chi^{\op}}}[K_0(\E)]$-module via
\[\alpha.\gamma = \zeta^{\chi^{\op}(\alpha,\gamma) + \chi^{\sigma}(\alpha)}(\HH(\alpha)+\gamma) \mbox{ for } \alpha \in K_0(\E),\ \gamma \in GW_0(\E),\]
where $\HH\!: K_0(\E) \to GW_0(\E)$, $[A] \mapsto [A \oplus A^{\vee}]$, is the hyperbolic plane map. Finally, by pulling back this module structure along the integral $\int_{\E} -\ d\mu_{\T}$, it becomes a $\H(\E,\T)$-module.

The proof proceeds along the very same lines as that of Theorem \ref{intthm}. First, we consider the square of correspondences expressing the compatibility of module structures, together with the fibre product in the left square,
\[\begin{tikzcd}
  \S_1(\E) \times_K (E\S_0(\E))^{\sigma} \ar[dd, "{\oint\times_K\oint^{\sigma}}"'] & & \ar[ll, "{(\lambda_1,\del_0)}"'] (E\S_1(\E))^{\sigma} \ar[r, "{\del_1}"] \ar[dl, dashed, "{\delta}"'] \ar[dd, "{\oint^{\sigma,(2)}}"] & (E\S_0(\E))^{\sigma} \ar[dd, "{\oint^{\sigma}}"] \\
   & \ar[dl, phantom, "{\pullbackop}"' very near start] \ar[ul, "{\epsilon}", "{\sim}"'] \Z \ar[dr, "{\sigma}"'] & & \\
  \ubar{K_0(S_1(\E))} \times_K \ubar{GW_0(ES_0(\E))} & & \ar[ll, "{(\ell_1,d_0)}", "{\sim}"'] \ubar{GW_0(ES_1(\E))} \ar[r, "{d_1}"'] & \ubar{GW_0(ES_0(\E))}
\end{tikzcd}\]
where the isomorphism $(\ell_1,d_0)$ is a special case of the additivity theorem \cite{Schl}, Theorem 3.2. Again, we conclude that there is an obstruction given by $\delta_*\delta^*$, which we therefore have to show to be locally linear. The fibre of $\delta$ over $(A,C)$ parametrizes diagrams as follows,
\[\begin{tikzcd}
  0 \ar[r, tail] & A \ar[r, tail] \ar[d, two heads] \ar[dr, phantom, "{\Box}"] & A^{\perp} \ar[r, tail] \ar[d, two heads] \ar[dr, phantom, "{\Box}"] & E \ar[d, two heads] & \\
   & 0 \ar[r, tail] & C \ar[r, tail] \ar[d, two heads] \ar[dr, phantom, "{\Box}"] & A^{\perp\vee} \ar[d, two heads] \\
   & & 0 \ar[r, tail] & A^{\vee} \ar[d, two heads] \\
   & & & 0
\end{tikzcd}\]
and therefore, again by similar arguments as in Theorem \ref{intthm}, is equivalent to
\[[\AA^{\dim\Ext^1(C,A)}_K/\Ga^{\dim\Hom(C,A)}] \times_K [\AA^{\dim\Ext^1(A^{\vee},A)^{\sigma}}_K\!/\Ga^{\dim\Hom(A^{\vee},A)^{-\sigma}}],\]
since the subspace of $\Ext^1(A^{\vee},A^{\perp})$ where $E$ is self-dual is a $\Ext^1(A^{\vee},A)^{\sigma}$-torsor. Thus,
\[(\delta_{\alpha,\gamma})_*(\delta_{\alpha,\gamma})^* = \zeta^{\chi^{\op}(\alpha,\gamma)+\chi^{\sigma}(\alpha)}\cdot\id,\]
where $\delta_{\alpha,\gamma}$ is defined analogously to \eqref{deltaab}, which completes the proof.

The next step is to set up a general formalism of self-dual slope filtrations to obtain interesting recursions to integrate and resolve to compute motivic invariants of the corresponding moduli stacks of semistable self-dual objects. This is done in the case of quiver representations in \cite{Matt3}.

\bigskip

\bibliography{higherS}
\bibliographystyle{plain}

\bigskip

\end{document}

%% file: header.tex
\usepackage[utf8]{inputenc}
\usepackage{xr-hyper}
\usepackage[colorlinks=false]{hyperref}
\usepackage{pdfsync}
\usepackage{verbatim}
\usepackage{amsthm,amsfonts,amssymb,mathtools}
\usepackage{xcolor}
\usepackage{graphicx}
\usepackage{tikz-cd}
\usepackage{enumitem}
\usepackage{stmaryrd}
\usepackage{bbm}
\usepackage{relsize}

\newcommand{\A}{\mathcal{A}}
\newcommand{\B}{\mathcal{B}}
\newcommand{\C}{\mathcal{C}}
\newcommand{\D}{\mathcal{D}}
\newcommand{\E}{\mathcal{E}}
\newcommand{\F}{\mathcal{F}}

\renewcommand{\H}{\mathcal{H}}

\newcommand{\K}{\mathcal{K}}

\newcommand{\M}{\mathcal{M}}
\newcommand{\N}{\mathcal{N}}
\renewcommand{\O}{\mathcal{O}}
\renewcommand{\P}{\mathcal{P}}
\newcommand{\Q}{\mathcal{Q}}
\newcommand{\R}{\mathcal{R}}
\renewcommand{\S}{\mathcal{S}}
\newcommand{\T}{\mathcal{T}}
\newcommand{\U}{\mathcal{U}}
\newcommand{\V}{\mathcal{V}}

\newcommand{\X}{\mathfrak{X}}

\newcommand{\Z}{\mathcal{Z}}

\newcommand{\h}{\mathfrak{h}}

\renewcommand{\l}{\ell}

\newcommand{\s}{\mathfrak{s}}

\renewcommand{\AA}{\mathbb{A}}

\newcommand{\CC}{\mathbb{C}}

\newcommand{\FF}{\mathbb{F}}
\newcommand{\GG}{\mathbb{G}}
\newcommand{\HH}{\mathbb{H}}

\newcommand{\LL}{\mathbb{L}}

\newcommand{\NN}{\mathbb{N}}

\newcommand{\PP}{\mathbb{P}}
\newcommand{\QQ}{\mathbb{Q}}
\newcommand{\RR}{\mathbb{R}}

\newcommand{\VV}{\mathbb{V}}

\newcommand{\ZZ}{\mathbb{Z}}

\newcommand{\AbCat}{\operatorname{AbCat}}

\newcommand{\Aff}{\operatorname{Aff}}

\newcommand{\an}{\operatorname{an}}
\newcommand{\AnSt}{\operatorname{AnSt}\!/ }
\newcommand{\AnSta}{\operatorname{AnSta}\!/ }
\newcommand{\AnVar}{\operatorname{AnVar}\!/ }
\newcommand{\Aut}{\operatorname{Aut}}

\newcommand{\Bun}{\operatorname{Bun}}
\newcommand{\Cat}{\operatorname{Cat}}
\newcommand{\cha}{\operatorname{char}}

\newcommand{\Coh}{\operatorname{Coh}}

\newcommand{\coker}{\operatorname{coker}}

\newcommand{\cris}{\operatorname{cris}}

\newcommand{\End}{\operatorname{End}}

\newcommand{\et}{\operatorname{\acute{e}t}}

\newcommand{\Ext}{\operatorname{Ext}}
\newcommand{\Fil}[2]{\operatorname{Fil}^{#1}_{#2}}

\newcommand{\Fun}{\operatorname{Fun}}
\newcommand{\Gal}{\operatorname{Gal}}
\newcommand{\GL}{\operatorname{GL}}

\newcommand{\Gr}{\operatorname{Gr}}
\newcommand{\gr}{\operatorname{gr}}

\newcommand{\Grpd}{\operatorname{Grpd}}
\newcommand{\Hom}{\operatorname{Hom}}
\newcommand{\HHom}{\H\!\operatorname{om}}
\newcommand{\id}{\operatorname{id}}
\newcommand{\im}{\operatorname{im}}

\newcommand{\Ind}{\operatorname{Ind}}

\newcommand{\isoc}{\operatorname{isoc}}

\newcommand{\Mod}{\operatorname{Mod}}

\newcommand{\op}{\operatorname{op}}

\newcommand{\Perf}{\operatorname{Perf}}

\newcommand{\Pic}{\operatorname{Pic}}
\newcommand{\pr}{\operatorname{pr}}
\newcommand{\pt}{\operatorname{pt}}

\newcommand{\proj}{\operatorname{proj}}
\newcommand{\qcoh}{\operatorname{qcoh}}

\newcommand{\refl}{\operatorname{refl}}
\newcommand{\Rep}{\operatorname{Rep}}
\newcommand{\rep}{\operatorname{rep}}

\newcommand{\Rex}{\operatorname{Rex}}
\newcommand{\rk}{\operatorname{rk}}

\newcommand{\sep}{\operatorname{sep}}

\newcommand{\SmVar}{\operatorname{SmVar}\!/ }

\newcommand{\Spec}{\operatorname{Spec}}

\newcommand{\ssi}{\operatorname{ssi}}
\newcommand{\St}{\operatorname{St}\!/ }
\newcommand{\Sta}{\operatorname{Sta}\!/ }

\newcommand{\Sym}{\operatorname{Sym}}

\newcommand{\tors}{\operatorname{tors}}

\newcommand{\val}{\operatorname{val}}
\newcommand{\Var}{\operatorname{Var}\!/ }
\newcommand{\Vect}{\operatorname{Vect}}
\newcommand{\vect}{\operatorname{vect}}

\renewcommand{\bar}{\overline}
\newcommand{\colim}{\varinjlim}

\newcommand{\dashto}{\dashrightarrow}
\newcommand{\del}{\partial}
\newcommand{\disj}{\amalg}
\newcommand{\dom}{\trianglerighteq}
\renewcommand{\epsilon}{\varepsilon}

\renewcommand{\hat}{\widehat}
\renewcommand{\iff}{\Leftrightarrow}

\newcommand{\into}{\hookrightarrow}

\newcommand{\onto}{\twoheadrightarrow}

\newcommand{\isom}{\cong}
\renewcommand{\lim}{\varprojlim}

\newcommand{\longiff}{\Longleftrightarrow}

\newcommand{\longinto}{\lhook\joinrel\longrightarrow}
\newcommand{\longonto}{\relbar\joinrel\twoheadrightarrow}
\newcommand{\longto}{\longrightarrow}

\newcommand{\longxto}[1]{\xrightarrow{\ #1\ }}
\newcommand{\minus}{\smallsetminus}
\renewcommand{\mod}{\bmod}
\renewcommand{\phi}{\varphi}
\renewcommand{\tilde}{\widetilde}
\newcommand{\ubar}{\underline}
\newcommand{\xfrom}[1]{\xleftarrow{#1}}
\newcommand{\xto}[1]{\xrightarrow{#1}}

\newcommand{\bee}[1]{\operatorname{B}\! #1}

\newcommand{\btimes}{\operatorname{\raisebox{-1pt}{\ensuremath{\boxtimes}}}}
\newcommand{\bbmu}{\vphantom{l\mu}\smash{\mbox{\ensuremath{\raisebox{-0.59ex}{\ensuremath{l}}\hspace{-0.15em}\mu\hspace{-0.91em}\raisebox{-0.915ex}{\scalebox{2}{\ensuremath{\color{white}.}}}\hspace{-0.62em}\raisebox{0.75ex}{\scalebox{2}{\ensuremath{\color{white}.}}}\hspace{0.42em}}}}}

\newcommand{\Ga}{\mathbb{G}_a}
\newcommand{\Gm}{\mathbb{G}_m}
\newcommand{\hatotimes}{\mathbin{\hat{\otimes}}}

\newcommand{\isoto}{\xrightarrow{\ \raisebox{-2pt}[0pt][0pt]{\ensuremath{\mathsmaller{\sim}}}\ }}

\newcommand{\pair}[1]{\langle{#1}\rangle}
\newcommand{\pullback}{\mathlarger{\mathlarger{\mathlarger{\mathlarger{\cdot\!\!\lrcorner}}}}}
\newcommand{\pullbackop}{\mathlarger{\mathlarger{\mathlarger{\mathlarger{\llcorner\!\!\cdot}}}}}

\newcommand{\Qlbar}{\overline{\mathbb{Q}}_{\ell}}

\DeclareRobustCommand{\skipTOCentry}[5]{}

\newtheorem{thm}{Theorem}[section]

\newtheorem{cor}[thm]{Corollary}
\newtheorem{lemma}[thm]{Lemma}
\newtheorem{prop}[thm]{Proposition}
\theoremstyle{definition}
\newtheorem{defn}[thm]{Definition}
\newtheorem{ex}[thm]{Example}
\newtheorem{rem}[thm]{Remark}

\numberwithin{equation}{section}